\documentclass[12pt]{article}
\usepackage{amssymb,amsfonts,amsthm,amsmath}
\usepackage[all]{xy}
\usepackage{slashed} 
\usepackage{yfonts}
\usepackage{mathrsfs,pifont}

\usepackage[dvips]{graphics}

\usepackage[dvips]{pstcol}
\usepackage{pstricks,pst-plot,pst-node}
\newpsobject{showgrid}{psgrid}{subgriddiv=1,griddots=5,gridlabels=0pt}

\newcommand{\C}{\mathbb{C}}
\newcommand{\Z}{\mathbb{Z}}
\newcommand{\R}{\mathbb{R}}
\newcommand{\Q}{\mathbb{Q}}
\newcommand{\bB}{\mathbb{B}}

\newcommand{\cO}{\mathscr{O}}
\newcommand{\cL}{\mathscr{L}}
\newcommand{\cA}{\mathscr{A}}
\newcommand{\cB}{\mathscr{B}}
\newcommand{\cN}{\mathscr{N}}
\newcommand{\cV}{\mathscr{V}}
\newcommand{\Ovir}{\mathscr{O}_\textup{vir}}
\newcommand{\Kvir}{\mathscr{K}_\textup{vir}}

\newcommand{\cM}{\mathcal{M}}

\newcommand{\fm}{\mathfrak{m}}
\newcommand{\fg}{\mathfrak{g}}
\newcommand{\ft}{\mathfrak{t}}
\newcommand{\fC}{\mathfrak{C}}

\newcommand{\cNr}{\cN_\textup{vtx}}
\newcommand{\bx}{\textup{\ding{114}}}

\DeclareMathOperator{\Hilb}{Hilb}
\DeclareMathOperator{\Coh}{Coh}
\DeclareMathOperator{\Hom}{Hom}

\DeclareMathOperator{\Ann}{Ann}
\DeclareMathOperator{\End}{End}
\DeclareMathOperator{\Aut}{Aut}
\DeclareMathOperator{\Lie}{Lie}
\DeclareMathOperator{\Iso}{Iso}
\DeclareMathOperator{\Def}{Def}
\DeclareMathOperator{\Obs}{Obs}
\DeclareMathOperator{\Chow}{\mathsf{Chow}}
\DeclareMathOperator{\Conf}{\mathsf{Conf}}
\DeclareMathOperator{\PT}{\mathsf{PT}}
\DeclareMathOperator{\DT}{\mathsf{DT}}
\DeclareMathOperator{\Rep}{Rep}
\DeclareMathOperator{\cycle}{\mathsf{cycle}}
\DeclareMathOperator{\sweep}{\mathsf{sweep}}
\DeclareMathOperator{\Ker}{Ker}
\DeclareMathOperator{\Coker}{Coker}
\DeclareMathOperator{\ind}{\mathbf{index}}
\DeclareMathOperator{\supp}{supp}
\DeclareMathOperator{\Pic}{Pic}
\DeclareMathOperator{\eval}{eval}
\DeclareMathOperator{\vir}{vir}
\DeclareMathOperator{\tr}{tr}
 
\DeclareMathOperator{\rk}{rk} 
 
\DeclareMathOperator{\Td}{Td}
\DeclareMathOperator{\ch}{ch}

\DeclareMathOperator{\Alt}{\textsf{Alt}}

\newcommand{\sind}{\mathsf{ind}}

\newcommand{\bL}{\mathbb{L}}
\newcommand{\pP}{\mathbb{P}}
\newcommand{\cC}{\mathscr{C}}
\newcommand{\cE}{\mathscr{E}}
\newcommand{\cF}{\mathscr{F}}
\newcommand{\cI}{\mathscr{I}}
\newcommand{\cJ}{\mathscr{J}}
\newcommand{\cG}{\mathscr{G}}
\newcommand{\cU}{\mathscr{U}}
\newcommand{\cX}{\mathscr{X}}
\newcommand{\Ct}{\mathbb{C}^\times}
\newcommand{\Cq}{\Ct_q}
\newcommand{\MM}{\textsf{M2}}
\newcommand{\tO}{\widetilde{\mathscr{O}}}
\newcommand{\bT}{\mathsf{T}}
\newcommand{\bS}{\mathsf{S}}

\newcommand{\bA}{\mathsf{A}}
\newcommand{\bV}{\mathsf{V}}
\newcommand{\TK}{\bT_\textup{K\"ahler}}
\newcommand{\pt}{\textup{pt}}
\newcommand{\bPhi}{\mathbf{\Phi}}
\newcommand{\cK}{\mathscr{K}}
\newcommand{\Dir}{\slashed D} 
\newcommand{\bk}{\boldsymbol{\kappa}}

\newcommand{\frL}{\textfrak{L}}

\newcommand{\Hlb}{\mathbf{H}}
\newcommand{\dbar}{\bar\partial} 
\newcommand{\cW}{\mathscr{W}}
\newcommand{\Ld}{{\Lambda}^{\!\raisebox{0.5mm}{$\scriptscriptstyle \bullet$}\,}}
\newcommand{\Ed}{{\cE}^{\raisebox{0.5mm}{$\scriptscriptstyle \bullet$}\,}}
\newcommand{\Sd}{{\bS}^{\raisebox{0.5mm}{$\scriptscriptstyle \bullet$}}}
\newcommand{\tSd}{{\widetilde{\bS}}^{\raisebox{0.5mm}{$\scriptscriptstyle \bullet$}}}
\newcommand{\oSd}{{S}^{\raisebox{0.5mm}{$\scriptscriptstyle \bullet$}}}
\newcommand{\Hd}{{H}^{\raisebox{0.5mm}{$\scriptscriptstyle \bullet$}}}

\newcommand{\Indx}{\underline{\mathit{Index}}}
\newcommand{\Gk}{G_{\bk}}

\newtheorem{Theorem}{Theorem}
\newtheorem{Conjecture}{Conjecture}
\newtheorem{Lemma}{Lemma}[section]
\newtheorem{Proposition}[Lemma]{Proposition}

\theoremstyle{definition}
\newtheorem{Definition}{Definition}

\begin{document}

\title{Membranes and Sheaves} 
\author{Nikita Nekrasov and 
Andrei Okounkov}
\date{April 2014} 
\maketitle

\setcounter{tocdepth}{2}
\tableofcontents

\section{A brief introduction}

\subsection{Overview}

Our goal in this paper is to discuss a conjectural correspondence between 
enumerative geometry of curves in Calabi-Yau 5-folds $Z$ and 
$1$-dimensional sheaves on 3-folds $X$ that are embedded in $Z$ as fixed
points of certain $\Ct$-actions. In both cases, the enumerative
information is taken in equivariant $K$-theory, where the 
equivariance is with respect to all automorphisms of the problem.

In Donaldson-Thomas theories, one sums up over all Euler characteristics
with a weight $(-q)^\chi$, where $q$ is a 
parameter\footnote{Note the difference with the traditional weighing by $q^\chi$ 
as in \cite{mnop}. The change of sign of $q$ fits much better with 
all correspondences.}.  Informally, $q$ is referred to as 
the \emph{boxcounting} parameter.  The main feature of the 
correspondence is that the 3-dimensional boxcounting parameter
$q$ becomes in $5$ dimensions the \emph{equivariant} parameter for $\Ct$-action that 
defines $X$ inside $Z$.  To stress this we will use the notation 
$\Cq$ in what follows.  

The $5$-dimensional theory effectively sums up the $q$-expansion in 
the Donaldson-Thomas theory. In particular, it gives a natural explanation
of the rationality (in $q$) of the DT partition functions. Other expected
as well as unexpected symmetries of the DT counts follow naturally from the 
5-dimensional perspective, see below.  These involve choosing different
$\Ct$-actions on the same $Z$ as our $\Cq$, and thus relating the 
same 5-dimensional theory to different DT problems. 

The important special case $Z=X \times \C^2$ is 
considered in detail in Sections \ref{s_refined} and
\ref{s_index_vertex}.  If $X$ is a toric Calabi-Yau threefold, we
compute the theory in terms of a certain \emph{index vertex}. We 
show the refined vertex found combinatorially by 
 Iqbal, Kozcaz, and Vafa in \cite{IKV} is a special case of the 
index vertex. 

\subsection{Motivation from M-theory}

\subsubsection{}

The aim of this section is to explain the physical origins of the 
problems studied in the paper and to give an interested 
physicist an idea of what is going on in this largely purely 
mathematical paper. 

One of the most striking features of the 
duality between string theory and M-theory is the geometric 
interpretation that it gives to the string coupling constant. Recall that 
the string coupling constant measures the amplitude of creating a
handle in the string worldsheet (which in the point particle limit
becomes the Planck constant, the weight of a Feynman diagram loop). 
String theory on a 
10-dimensional\footnote{We use real dimensions until we specialized the 
discussion to complex manifolds; complex dimensions are used
elsewhere in the paper.}   spacetime $Z$ is related to
M-theory on an circle-bundle 
\begin{equation}
  \xymatrix{
S^1 \ar@{^{(}->}[r] &  \widetilde{Z} \ar@{->}[d] \\
& Z 
}\label{tZ} 
\end{equation}
over $Z$, and the 
length of the circle fiber translates into the string coupling
constant \cite{Witten:1995ex}. 

When the 10-dimensional spacetime $Z$ is a product 
$$
Z = X \times {\R}^{1,3}
$$
of a Calabi-Yau threefold $X$ and the Minkowski space ${\R}^{1,3}$,
certain string theory amplitudes describing the scattering of soft
graviphoton modes in the effective four dimensional supergravity
theory are given exactly by the genus $g$ amplitudes of the
topological string theory on $X$ \cite{Bershadsky:1993cx,
  Antoniadis:1993ze}.  In this computation, the role of 
of string coupling constant is replaced by the 
field strength of the graviphoton gauge field. Topological 
string amplitudes have an accepted mathematical definition as 
Gromov-Witten invariants of $X$. 

\subsubsection{}

The appearance of the Donaldson-Thomas theory of $X$ may be traced to 
the duality between between the Taub-NUT space $(\R^4, ds_\textup{TN}^2)$,
also known as the
Kaluza-Klein magnetic monopole, in M-theory and the D6-brane of the
IIA string proposed in \cite{Hull:1997kt,Sen:1997js}.  

Recall that $ds_\textup{TN}^2$ is a complete hyperK\"ahler metric
on $\R^4$ with $U(2)$ group of isometries. In particular, the fibers
of any rank 2 holomorphic bundle $\cV$ over a K\"ahler manifold $X$ 
may be given the Taub-NUT metric. For 
\begin{equation}
\widetilde{Z} = \!\!\!\!
\begin{matrix}
\quad\,\, \cV \oplus  \R^1_\textup{time} \\
\downarrow\\
X
\end{matrix} \,,\label{tZ2}
\end{equation}
to be a 
suitable background for M-theory it is necessary, in particular, that 
$$
\det \cV = \cK_X
$$
where $\cK_X$ is the canonical bundle of $X$. This means 
$\cK_Z$ is trivial, where $Z$ is 
the total space of $\cV$. 

The D6-brane emerges when we use $U(1) \subset SU(2) \subset 
\textup{Iso}(ds_\textup{TN}^2)$ as
the circle in \eqref{tZ}. For such $U(1)$-action to exist globally, 
we assume a decomposition $\cV = \cL_1 \oplus \cL_2$ into a 
direct sum of two line bundles with $\cL_1 \otimes \cL_2 = \cK_X$. 
The dual string description is that of IIA string on 
$$
Z' = 
\begin{matrix}
\cK_X \oplus \R^{1,1}  \\
\downarrow\\
X
\end{matrix} \,, 
$$
with a single D6-brane wrapped on $X$. 

On this D6-brane, lives a $U(1)$ gauge theory with maximal
supersymmetry in flat space-time ${\R}^{1,6}$. The bosonic fields of
the gauge supermultiplet are a gauge field $A$ and a triplet of
scalars ${\vec \Phi}$. When D6 is wrapped on $X$, the corresponding
supersymmetric theory is
twisted  in such a way that the triplet of scalars become a real
scalar ${\sigma} = {\sigma}^{*}$ and a complex field $\phi$, a section  
of the canonical bundle $\cK_X$. 

This theory on $X$ is not the conventional 
abelian gauge theory. In some aspects, for example,
 it can be viewed as a noncommutative deformation of the $U(1)$ gauge
 theory \cite{Nekrasov:1998ss}, the non-commutativity being related to
 the choice of the $B$-field in the IIA picture 
 \cite{Witten:2000mf}. Depending on the stability parameters, the 
instantons of this $U(1)$ theory can be, for example,
ideal sheaves $X$, see e.g.\ Section 6 in \cite{Iqbal:2003ds}. 

One of the goals of this paper is the precise 
identification of this theory with $K$-theoretic Donaldson-Thomas
theory of $X$. This identification takes into account both the 
intrinsic geometry of $X$, which \emph{need not} to be 
Calabi-Yau, and the extrinsic geometry of $X$ in $Z$, which is 
specified by the choice of $\cL_i$. The general context of our 
proposed identification is when $Z$ is a (noncompact) 
Calabi-Yau 5-fold and $X\subset Z$ is a fixed locus of a 
$U(1)$-action\footnote{
Of course, 
here $U(1)$ may be replaced by its complexification 
$\Ct$.}
that preserves the 5-form on $Z$. This generalizes to many other geometries in which the fixed locus
$X$ may be disconnected, such as those corresponding to 
multi-center Taub-NUT metrics, see e.g.\ Section \ref{eng_Hr}. 

The more familiar cohomological Donaldson-Thomas theory, 
which in particular is conjectured to be equivalent to the Gromov-Witten 
theory and thus to topological strings \cite{mnop}, 
is a certain limit from the K-theoretic 
computations that we do in this paper. 

\subsubsection{}

The unbroken supersymmetry of the theory on $X$ 
may be interpreted as the Dirac operator acting in a 
certain infinite-dimensional space\footnote{It is probably useful to keep in mind the analogy with elliptic genus \cite{Witten:1986bf}. There, the Dirac operator on the loop space of a Riemannian manifold $M$ is the unbroken supercharge of the two-dimensional sigma model with the target space $M$. In our case, the unbroken supercharge is the Dirac operator on the space of gauge equivalence classes of the pairs $(A, {\phi})$ consisting of the six dimensional gauge field $A$ and the $\cK_X$-valued
Higgs field. See also \cite{Nekrasov:1996cz} for the analogous
discussion in $4+1$ dimensions}. The index of this operator is the
partition function of the theory when time is made periodic, with 
periodic boundary conditions for fermions, also
known as Witten index. This index is best treated in equivariant 
K-theory with respect to all automorphism of the problem, which  
corresponds to taking more general 
quasi-periodic boundary conditions in time. 

Because the same time periodicity may be imposed in \eqref{tZ2}, 
it is reasonable to expect that the K-theoretic DT 
index of $X$ equals the index of $M$-theory on $Z$
(cf. \cite{Losev:1997hx}). More precisely, since the instantons 
in DT theory may be seen as bound states of D6, D2, and D0 branes, 
this is the index of the sector that involves only membranes of 
M-theory and excludes the M5-branes. Finding a Donaldson-Thomas
description of M5-branes remains an important direction 
for future research.


\subsubsection{}

With a purely algebro-geometric description of the theory on the 
D6 brane at hand, it is logical to ask for a similar description of
membrane contributions to the M-theory index. Ideally, the 
moduli space of supersymmetric membranes should be described
as a compact algebraic variety for any given homology class of the 
membrane and the contribution of these membranes to the
M-theory index should equal the index of a certain canonical 
K-theory class on this moduli space.

It is natural to pursue this goal for an 
arbitrary smooth quasiprojective 5-fold $Z$ with a trivial canonical 
bundle $\cK_Z$. In particular, we don't require $Z$ to be compact or 
satisfy other constraints like those in \cite{HLS}. 
A prototypical supersymmetric membrane in this geometry 
has the form $S^1_\textup{time} \times C$, where $C\to Z$ is an immersed 
algebraic curve. 

The geometric and physical difficulty arises when $C$ degenerates
and develops multiplicities or other singularities.
Moduli spaces typically used 
in algebraic geometry are not suitable either because their 
local geometry, i.e.\ the deformation 
theory, is too bad (e.g.\ Chow varieties) or because they have 
infinitely many connected 
components for a fixed degree of $C$ (e.g.\ moduli of
stable maps), or both (e.g.\ Hilbert schemes). 

Both issues are
problematic for a physicist who wants to construct a version of
Dirac operator on these moduli spaces
and does not have parameters to keep track of discrete
invariants of $C$ other than its degree. 
While understanding multiple 
membranes has been a very active area of research, see for example 
\cite{BLMP}, it is not clear to us what the
approaches surveyed there say about the geometric problem at
hand. 

Based on our conjectural correspondence with Donaldson-Thomas theory, 
we make a proposal for the membrane moduli spaces, the pros and 
cons of which are discussed in Section \ref{s_M}. In any event, 
we expect our Conjecture \ref{sconj1} to be very useful as a selection tool between various candidates for
moduli of supersymmetric membranes.

\subsection{Plan of the paper}

Section \ref{s1} discusses the general outline of the conjectures, 
without a complete specification of the integrands. Those are
discussed in Section \ref{s_DT} for the Donaldson-Thomas theory 
and Section \ref{s_M} for membranes of M-theory, respectively. 
Several examples of the correspondence, in which one can already 
see all the ingredients of the general conjectures, are discussed
in Section \ref{s_exa}. 

The integrands in both Donaldson-Thomas theory and 
M-theory involve square roots of certain line bundles. The 
existence of these square roots is investigated in Section \ref{s_sqrt}. 

A very special case of the general theory is when $X$ is Calabi-Yau
and $Z=X \times \C^2$. In this case, the index of DT theory 
enjoys a certain rigidity: it factors through  a 
 \emph{character} of the automorphism group given 
by the square root of the weight of $\Omega^3_X$. This is 
discussed in Section \ref{s_refined}.  This rigidity simplifies
computations. In particular, for an arbitrary toric threefold, 
the K-theoretic DT invariants may be expressed in terms of 
a certain K-theoretic vertex, see Section \ref{s_K_vertex}. 
For $Z=X \times \C^2$, we can replace it by 
a simpler object, the index vertex, see Section \ref{s_index_vertex}.

\subsection{Acknowledgements}

We have beed working on this project for a long time and our 
interaction with Mina Aganagic,  Johan de Jong, Davesh Maulik, Edward Witten, and
others were very important for us in this process. 

Davesh Maulik made a decisive progress on 
 our conjecture relating $K$-theoretic and 
motivic Donaldson-Thomas invariants of Calabi-Yau 3-folds with a
torus action scaling the $3$-form, see \cite{Dmot}. Various computations 
with our index vertex are discussed by Choi, Katz, and Klemm in 
in \cite{CKK}. 

Research of NN was supported in part by RFBR grants 12-02-00594, 12-01-00525, by Agence Nationale de Recherche via the grant
ANR 12 BS05 003 02,  by Simons Foundation, and by Stony Brook
Foundation. Research of AO was supported by NSF FRG 1159416. 

A major part of this work was done while AO was visiting 
Simons Center for Geometry and Physics in Augusts of 2011 
and also 2012 
and he wishes to thank 
the Center and its director John Morgan for warm hospitality. 
We also thank IHES, Princeton University, Imperial College London,
MIT, and many other institutions where we had the opportunity 
to discuss the results presented here. 

We changed the 
preliminary title ``The index of M-theory'', which was also the 
title of many of our talks, to the one that, in our view, better 
reflects the essence of our main conjecture.

\section{Contours of the conjectures}\label{s1}

\subsection{$K$-theory preliminaries}

\subsubsection{}

In this paper, we use the word \emph{sheaf} as a shorthand for two very
different objects. The precise meaning should be 
clear from the context, except in the title of the paper. 

In most instances, by a sheaf on a scheme 
$Y$ we really mean a $K$-theory 
class of equivariant quasicoherent sheaves on $Y$. However, when 
we talk about moduli $\cM$ of sheaves on a smooth 3-fold $X$, we mean 
moduli of complexes of coherent sheaves on $X$ of specific shape 
and subject to certain stability conditions.  

The two 
occurrences of the word in the 
phrase \emph{let $\Ovir$ be the virtual structure sheaf of the moduli 
space $\cM$ of sheaves on $X$} exemplify the two different meanings. 

\subsubsection{}

For quasi-coherent sheaves
$\cF$, 
we require an action of a torus $\bT$ on $\cF$ such that: 
\begin{enumerate}
\item[(1)] $\bT$ acts trivially on $Y$;
\item[(2)] all weight spaces are coherent; 
\item[(3)] all nonzero weight spaces lie in a translate of a fixed nondegenerate
cone in the character group $\bT^\vee$. 
\end{enumerate}
The last condition makes sure $K_\bT(Y)$ is a ring with respect 
to tensor product and a module over 
$$
\Rep \bT = K_\bT(\pt)\,,
$$ 
which is defined with the same cone support condition.  

\subsubsection{}

The equivariance is always assumed to 
be maximal possible, i.e.\ with respect to all symmetries of the problem. 
For example, once a subgroup $\Cq\in \Aut(Z,\Omega^5)$ has been fixed, 
we want all constructions to be equivariant with respect to its centralizer 
$$
G_q = \Aut(Z,\Omega^5)^{\Cq} 
$$

\subsubsection{}\label{coneCq} 

A simple but fundamental choice for everything in the paper is the 
choice of the cone in $\Z = \left(\Cq\right)^\vee$. In English, it is 
a choice between expanding rational functions on $\Cq$ in a series near
$q=0$ or $q=\infty$. 

We choose $\Z_{\ge 0} \subset \Z$, or, equivalently, 
we choose expansions in ascending powers of $q$. This choice is
reflected in the asymmetry with which the attracting and repelling 
direction for the $\Cq$-action enter the formulas below.  

\subsubsection{}

To keep track of degree of curves in $Z$, it 
convenient to formally introduce a torus
\begin{equation}
\TK(Z) = \frac{H^2(Z,\C)}{2\pi i \, H^2(Z,\Z) \big/ \textup{torsion}} \cong (\Ct)^{b_2(Z)} \,. 
\label{defTK}
\end{equation}
By construction 
$$
\TK(Z)^\vee = H_2(Z,\Z) \big/ \textup{torsion} 
$$
so any curve $C\subset Z$ defines a character of $\TK(Z)$ which 
we denote $Q^{[C]}$. A natural nondegenerate cone in $\TK(Z)^\vee$ 
is formed by classes of holomorphic curves. 

\subsubsection{}\label{s_K_push}

All pull-backs and push-forwards are 
taken in equivariant $K$-theory. Non-proper push-forwards are defined
as equivariant residues if the induced maps on torus-fixed points are 
proper.

\subsection{The index sheaf} 

\subsubsection{} 

Let $Z$ be a nonsingular algebraic $5$-fold with a nowhere vanishing 
holomorphic $5$-form $\Omega^5$.  For any $g\in \Aut(X,\Omega^5)$
the following $Z$-bundle over $S^1$
\begin{equation}
S^1 \rtimes_g  Z = \R \times Z \Big/ (t,z)\sim (t+\ell,g\cdot z) \label{eSZ}
\end{equation}
is an 11-manifold on which M-theory may be studied. Here $\ell \in \R$
is a parameter, the length of the M-theory circle. 

 From general 
principles, 
$$
\textup{Partition function}(S^1 \rtimes_g  Z) = \tr_{\textup{Hilbert
    space}}
\, (\pm 1)^F g \, \exp\left(\ell \, \frac{d}{dt}\right)
$$
where $F$ is the fermion number operator and $g$ and $\frac{d}{dt}$ 
denote the action of the symmetry $g$ and an infinitesimal time
translation on the Hilbert space of the theory. The sign in $ (\pm
1)^F$ depends on the boundary conditions for fermions $\psi$.  In what 
follows, we choose  $(-1)^F$ which corresponds to 
$$
\psi(t+\ell) = \psi(t)\,. 
$$
With this choice of sign, supersymmetry will cancel all contributions to the partition 
function except for a certain index, known as the Witten index in this
context. 

\subsubsection{}

Supersymmetry means that the infinitesimal space time translation is
the square of an odd operator
\begin{equation}
\frac{d}{dt} = \Dir^2\,, \label{Dir2}
\end{equation}
which is a certain infinite-dimensional version of the Dirac
operator. 

While our understanding of the kinematics and dynamics
of M2-branes is still in its infancy, we may reasonably expect 
$\Dir$ to resemble Dirac operators familiar from 
finite-dimensional supersymmetric quantum mechanics on K\"ahler 
manifolds, see \cite{qfs}, which we briefly recall. 

 In particular, as a formal consequence of 
$$
\left\{\Dir, (-1)^F \right\}= 0 
$$
one expects
$$
\textup{Partition function}(S^1 \rtimes_g  Z) = 
\tr_{\ind
    \Dir} g 
$$
where
$$
\ind \Dir = \left(\Ker \Dir\right)_\textup{even} -  \left(\Ker
  \Dir\right)_\textup{odd} 
$$
is a virtual representation of all symmetries of the theory. 

\subsubsection{}

Let $\Conf$ denote the configuration space of a finite-dimensional 
classical mechanical system. This is a Riemannian manifold with the metric determined by 
kinetic energy. 

Hilbert spaces $\Hlb$ of corresponding quantum systems are formed by sections of
certain line bundles $\frL$ over $\Conf$. Differential operators acting in $\Hlb$
form a quantization of functions on $T^*\!\Conf$, that is, a quantization 
of the algebra of classical observables.  This algebra contains the Hamiltonian, 
i.e.\ the action of the infinitesimal time translation 
\begin{equation}
\frac{d}{dt}  \mapsto \textup{const} \, \Delta + \dots  \in \End \Hlb\,,
\label{HamLap}
\end{equation}
where $\Delta$ is the Laplace operator and dots stand for a differential operator 
of lower order \footnote{Nonzero constant like the one in \eqref{HamLap} 
are irrelevant for 
index computations and we will not pay attention to them.}. 

\subsubsection{}

To add fermions, one introduces a vector bundle $\Psi$ over $\Conf$ and takes
$$
\Hlb = L^2(\Conf,\frL \otimes \Ld \Psi^*)\,. 
$$
Sections of $\Ld \Psi^*$ may be viewed as function on a configuration
supermanifold where the odd degrees of freedom are described by the
bundle $\Psi$. 
Sections of $\Psi$ and $\Psi^*$ act by fermionic annihilation and creation operators, 
respectively, on the exterior algebra $\Ld \Psi^*$.

 In special cases, 
the  square root \eqref{Dir2} exist. For example, if $\frL$ is flat and 
$\Psi$ is the tangent bundle, we can take 
$$
\Dir = d + d^* \,,
$$
where $d$ is the deRham differential on $\frL$-valued forms. The cohomology 
of $d$ is the cohomology of $\Conf$ with values in the local system $\frL$. 

\subsubsection{} 

K\"ahler configuration spaces admit enlarged supersymmetry and more Dirac operators. 
If the metric on $\Conf$ is K\"ahler then the splitting 
$$
T\Conf \otimes_\R \C = T^{1,0} \oplus T^{0,1} 
$$
is holonomy invariant and for any holomorphic bundle $\cE$ one can take 
$$
\Dir =  \dbar + \dbar^*
$$
where
$$
\dbar \in \End L^2(\Conf,\cE \otimes \Ld {T^{0,1}}^*)
$$
is the Dolbeault differential. 

Holomorphic bundles thus play the same role for 
K\"ahler manifolds as flat bundles play for general Riemannian manifolds, including 
the identification 
\begin{equation}
\ind \Dir= \chi(\cE)\,,
\label{indchiE}
\end{equation}
where $\chi(\cE)$ is the holomorphic Euler characteristic. 

\subsubsection{}\label{s_Spinor}

A particularly important special case is when $\cE$ is a line bundle that squares to 
the canonical bundle 
$$
\cE^{\otimes 2} = \cK_{\Conf} =  \Lambda^{\textup{top}} {T^{1,0}}^* \,, 
$$
in which case 
$$
\mathscr{S}_{\pm} = \cK_{\Conf}^{1/2} \otimes \Lambda^\textup{even/odd} \, {T^{0,1}}^*
$$
are the spinor bundles of $\Conf$.  Square roots of (virtual) canonical bundles 
will appear everywhere in this paper.

\subsubsection{}

Suppose 
$$
\Ed = \frL \otimes \Ld \Psi^*\,,
$$
where $\frL$ and $\Psi$ are holomorphic bundles and 
let $s$ be holomorphic section $s$ of $\Psi$. Contraction with $s$
defines Kozsul complex on $\Ed$ which is 
exact away from 
\begin{equation}
\cM = \{s = 0 \}  \subset \Conf \,. \label{defcM}
\end{equation}
We then may take
$$
\Dir = Q  + Q^* 
$$
where $Q$ is the differential in the total complex of the Dolbeault double complex
of $\Ed$. The equality \eqref{indchiE} still holds, where $\chi(\Ed)$ in now the 
Euler characteristic of a complex. 

Since $\|s\|^2$ enters as the potential term in the Hamiltonian $\Dir^2$, the 
submanifold $\cM$, formed by absolute minima of $\|s\|^2$, is also known 
as the locus of supersymmetric vacua in $\Conf$. 
The special case $s=\partial \cW$, where 
$$
\cW : \Conf \to \C 
$$
is a holomorphic function called \emph{superpotential}, is often emphasized. 

\subsubsection{}

The relevance of this discussion for systems with infinitely many degrees of 
freedom lies in the fact that even for infinite-dimensional $\Conf$ and $\Psi$, 
the complex $\Ed$ may turn out to be quasi-isomorphic, at least formally, to 
a bounded complex 
\begin{equation}
\Ed \cong \Indx \in D^b(\Coh \cM) \label{defIndx}
\end{equation}
of coherent sheaves supported on a countable disjoint union $\cM$ of algebraic 
varieties\footnote{In the present paper, we focus on the index, which only depends
on the $K$-theory class of the complex \eqref{defIndx}. However, the finer 
information lost by passing to the $K$-groups is of definite physical importance 
and it would be very interesting to know whether it can be accessed along the 
lines of the present paper.}. 

If, in fact, $\cM$ has infinitely many connected components then the theory 
must have a parameter that serves as the argument of the generating 
function over $\pi_0(\cM)$.  

\subsubsection{}

For M2-branes in $Z$, the configuration space $\Conf$ is the loosely defined
space of all surfaces in $Z$. It is reasonable to think it inherits the K\"ahler 
structure from that of $Z$. 

The moduli space $\cM$ of supersymmetric M2-branes is expected to be a 
certain compactification of the moduli space $\cM_0$ of immersed holomorphic 
curves $f:C\to Z$. For given degree
$$
\beta=f_* [C] \in H_2(Z,\Z) 
$$
and genus $g=g(C)$, $\cM_0$ is an algebraic variety with perfect obstruction 
theory given by 
$$
\Def - \Obs = \Hd(C,N_f) \,,
$$
where $N_f$ is the normal bundle to the immersion $f$. 

\subsubsection{}\label{Mbound} 

M-theory has a field, namely the 3-form, that couples to the degree $\beta$ 
through its 2-form component along $Z$.  This gives the variables in the 
K\"ahler torus \eqref{defTK} that grade the index by the degree of the membrane. 

A simple but essential point is that M-theory \emph{does not} have a parameter
that couples to the genus of $C$. A related observation is that Euler characteristic 
vanishes for any smooth real 3-fold, in particular, for 
a smooth worldvolume of an M2-brane. However, the genus of an 
immersed holomorphic curve is bounded above in terms of $\beta$, and hence 
a special genus-counting parameter is not required. 

Whether or not $\cM_0\subset\cM$ is dense, we will require $\cM$ to be 
an algebraic variety for fixed degree. 
This will insure that the grading by the K\"ahler torus $\TK(Z)$, with 
the assumptions of Section \ref{s_K_push}, is sufficient to define the M2-brane index.

\subsubsection{} 

The content of this paper may be very informally described as an attempt to guess
the space $\cM$, with the sheaf $\Indx$, from a mixture of constraints, clues, and 
conjectures, such as those just discussed. 

The principal new ingredient is a conjectural relation with Donaldson-Thomas (DT)
theory of algebraic $3$-folds that arise as fixed points $Z^{\C^\times}$ for certain 
special $\C^\times$-actions on $Z$. This relation will be discussed presently.

\subsection{Comparison with Donaldson-Thomas theory}

\subsubsection{} 
Our conjectural connection between M2-brane index and DT theory takes place 
when $Z$ admits a $\C^\times$-action of a very special kind. To distinguish this special 
$1$-dimensional torus from all other ones, we denote its element by $q$ and write
$\Cq$. 

So, we suppose there exists a symmetry
$$
\Cq \hookrightarrow \Aut(X,\Omega^5)
$$
such that it fixed locus 
$$
X = \bigsqcup X_i = Z^{\Cq} \,. 
$$
has pure dimension $3$. Here $X_i$ are the connected components of $X$. 
Since $\Cq$ preserves the $5$-form, we have 
\begin{equation}
N_X Z = \cL_1 \oplus \cL_2 \,, \quad \cL_1 \otimes \cL_2 = \cK_X\label{NXZ}
\end{equation}
where $\cL_1$ and $\cL_2$ are $\Cq$-eigensubbundles with weights 
$q$ and $q^{-1}$, where $q\in \Cq$ is the coordinate.

In particular, 
the total space of rank two bundle like \eqref{NXZ} over an 
arbitrary nonsingular 3-fold $X$ is the basic example for most
constructions in this paper. 

\subsubsection{}

Since each $X_i$ is a nonsingular 3-fold, its DT theory is defined.  In particular, the DT moduli spaces have virtual structure 
sheaves as well as modified virtual structure sheaves $\tO_{\DT}$ which will be discussed below.

Of the many possible stability chambers of the DT theory of $X$, the Pandharipande-Thomas chamber is the natural choice for us. The PT moduli spaces parameterize $1$-dimensional sheaves with a
section 
$$
s : \cO_X \to \cF\,,
$$
subject to certain stability conditions. In particular, these spaces are trivial in degree zero, 
matching the trivial contribution of empty membranes to the M-theory
index. 

Formula \eqref{ePTDT} below summarizes the expected relation between 
$K$-theoretic counts in the PT and the Hilbert scheme chambers.

\subsubsection{} \label{s_pi_DTM}

If $\cF$ is the 1-dimensional sheaf on $X$ we set
$$
\cycle(\cF) = \sum_{\cC \subset \supp \cF} \textup{length}(\cF_c) \,
\cdot \, \cC
$$
where $\cC$ ranges over 1-dimensional components 
of the reduced support of 
$\cF$ and $c\in \cC$ is the generic point. This may  be 
promoted to a morphism\footnote{
There is a large body of research on constructing the 
 parameter space for cycles
in $X$ of given dimension ($=1$, for us) and degree, 
first as a reduced algebraic variety, the Chow variety, see 
in particular \cite{Bar,Koll}, and then,
ideally, as a scheme with a natural scheme structure, such
that e.g.\  $\pi_{\PT}$ and $\pi_{\MM}$ are
maps of schemes. Certain aspects of this theory will be revisited in the 
forthcoming note \cite{JO}. We continue to call $\Chow(X)$ 
the Chow variety for historical reasons.}
$$
\pi_{\PT} : \PT(X) \to \Chow(X) 
$$
from the Pandharipande-Thomas moduli spaces of $X$ 
 the Chow variety of $X$. 
On the membrane side, there is a parallel map
$$
\pi_{\MM}: \MM(Z)^{\Cq}  \to \Chow(X)
$$
that keeps those components $C_i$ of $C = \bigcup C_i$ that are fixed point-wise
by $\Cq$ and discards the others, see Figure \ref{f_connectors}.

\begin{figure}[!htbp]
  \centering
   \scalebox{0.75}{\includegraphics{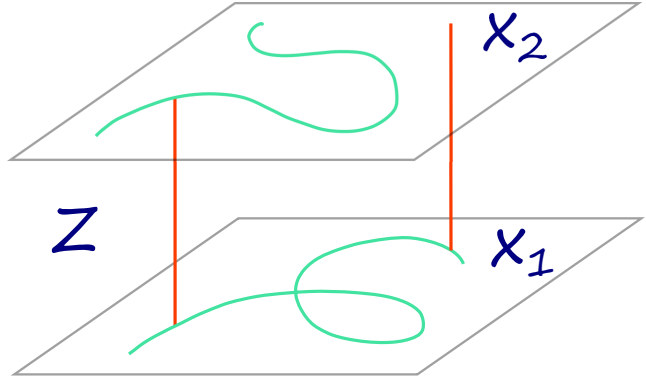}}
 \caption{The $\Cq$-fixed locus $X$ may be disconnected and the map
   $\pi_{\MM}$ keeps those components of $\Cq$-invariant curves that
   lie in $X$ and discards the $\Cq$-orbits that are drawn vertically
 in the picture.}
  \label{f_connectors}
\end{figure}

\subsubsection{}

Consider the diagram of maps
$$
\xymatrix{
\MM(Z) & \MM(Z)^{\Cq} \ar@{->}[l]_{\displaystyle{\iota}} \ar@{->}[rd]_{\displaystyle{\pi_\MM}}
&& \PT(X) 
\ar@{->}[ld]^{\displaystyle{\pi_{\PT}}}\\
&& \Chow(X) 
}
$$
in which $\iota$ is the inclusion of the fixed locus. 
Assuming an equivariant localization formula may be proven for
$\tO_\MM$, it would produce a sheaf $\tO_{\MM,\textup{localized}}$
on the fixed locus such that 
$$
\iota_* \tO_{\MM,\textup{localized}} =  \tO_\MM
$$
in localized equivariant $K$-theory of $\MM(Z)$. We denote
$$
 \tO_{\MM,\textup{localized}} = \iota_*^{-1} \,  \tO_\MM \,. 
$$
This puts us 
in the position to compare the push-forward of $\tO_{\MM,\textup{localized}}$
to the Chow variety of $X$ with the similar push-forward from the 
sheaf side. 

\subsubsection{}
There are natural $S(d)$-invariant maps 
$$
\Sigma_d: \Chow(X)^{\times d} \to \Chow(X)
$$
given by addition of cycles 
$$
\left(\cC_1,\dots,\cC_d\right) \mapsto \sum \cC_i \,.
$$
Given a sheaf $\cF$ on $\Chow(X)$, we define its symmetric 
algebra over $\Chow(X)$ by
$$
\bS_{\Chow}\,  \cF = \bigoplus_{d=0}^{\infty} \left( \Sigma_{d,*} \cF^{\boxtimes d}
\right)^{S(d)} \,. 
$$

\subsubsection{}

The following is our main conjecture, in an abstract form: 

\begin{Conjecture}\label{sconj1} We have the following equality in 
$\TK(Z) \times G_q$-equivariant $K$-theory of the Chow 
variety: 
\begin{equation}
\bS_{\Chow}\,  \pi_{\MM,*} \, \iota_*^{-1} \, \tO_{\MM} = \pi_{\PT,*} \, 
\left(
\tO_{\PT} \otimes \bPhi \right)\label{conj1}
\end{equation}
where $\bPhi$ is a certain explicit combination of the universal sheaves
on $\prod \PT(X_i)$ that describes the interaction of the 
components of $X$ 
inside $Z$, see Section \ref{s_inter} below. 
\end{Conjecture}

The modified virtual structure sheaves $\tO_{\MM}$ and $\tO_{\PT}$ 
are discussed in Sections \ref{s_DT} and \ref{s_M}, respectively. 
The interaction term $\bPhi$ has an explicit dependence 
on variables $Q^{[C_i]} \in \Rep \TK(Z)$ where $C_i$ is a $\Cq$-invariant
curve discarded by the map $\pi_{\MM}$. 

\subsubsection{}

There are numerous advantages to formulating our conjectures are
a comparison of sheaves on the Chow variety. 

Most importantly, 
in this paper we make only partial progress towards constructing 
the sheaf $\tO_{\MM}$. However, there is a good understanding 
of it over a large open set in the Chow variety and the corresponding 
statement \eqref{conj1}
is highly nontrivial and may be subjected to many checks.

Further, the construction of the modified virtual 
structure sheaves $\tO_{\MM}$ and $\tO_{\PT}$ requires 
finding square roots of certain line bundles. For these 
square roots to exist globally, one may need to introduce an 
additional twist by a line bundle pulled back 
from the Chow variety, see Section \ref{pull_square}. The formulation 
\eqref{conj1} avoids these 
complications modulo a certain 
technical provision\footnote{In principle, it can happen
that the moduli of $\Cq$-orbits \emph{discarded} by the map $\pi_{\MM}$ do 
not admit a square root of the virtual canonical bundle, see the 
discussion in Section \ref{s_23inter}.}.

\subsection{Fields of 11-dimensional supergravity and degree zero DT counts}

\subsubsection{}

M-theory is a quantum theory of gravity which is believed to reduce, at low energies, to the eleven dimensional supergravity. 

In this paper, we mostly focus on the contribution of membranes to the 
M-theory index. There is also a contribution of supergravity fields to the index
which, in principle, is easier to determine because of its local 
nature. A conjectural connection between the field index and 
degree zero K-theoretic Donaldson-Thomas invariants was 
discovered in \cite{Zth, ZthJ}.  Since this paper is a 
natural development of the ideas of \cite{Zth}, we summarize them
briefly. 

\subsubsection{}

We consider M-theory on a manifold of the form 
\eqref{eSZ} in the Hamiltonian formulation and linearized around 
a certain vacuum configuration. This means that 
as our configuration space we take 
$$
\Conf = \frac{\{ \textup{global sections of bosonic fields on $Z$}\}}
{\textup{gauge equivalence}  \sim} \,,
$$
where the linearized bosonic fields of the eleven dimensional
supergravity are
a small perturbations
$$
{\delta}g \in \Gamma(S^2 T^*Z) \,,
$$
 of some background metric $g_0$
and the $3$-form
$$
A \in \Gamma(\Omega^3 Z) \,.
$$
At the linearized level, the gauge equivalence classes are the 
cosets by the image of the vectors fields and 2-forms on $Z$ 
that act by infinitesimal diffeomorphisms and 
$$
A \mapsto A + d \omega\,, \quad \omega \in \Gamma(\Omega^2 Z)\,,
$$
respectively. In particular, $\Conf$ is an infinite-dimensional 
linear space (even for compact $Z$, since neither sections nor 
bundles are holomorphic at this point). In addition one imposes, in canonical gravity, the invariance under the diffeomorphisms of the eleven-dimensional space-time manifold. After the Diff$(Z)$ invariance is imposed, there is one more constraint, the so-called Hamiltonian constraint, which is a second order differential-variational equation to be obeyed by the allowed sections of the appropriate line bundle over $\Conf$. Instead of trying to solve this constraint, for the purposes of enumerating the solutions, it is sufficient to restrict the class of metric perturbations. A convenient choice is to impose the traceless constraint on ${\delta}g$: 
$$
{\tr} g_{0}^{-1}{\delta}g = 0
$$
where we used the background metric $g_0$ to make an operator
$g_{0}^{-1}{\delta}g : TZ \mapsto TZ$.

The isometries of $g_0$ act on $\Conf$ 
by linear operators. 

\subsubsection{}

While questions of regularity of sections, boundary conditions etc.\
are of paramount physical importance, index computations are typically 
less sensitive to such issues and in the present discussion they will
be ignored entirely. Our computations will be formally 
modeled on the following basic example. 

Suppose $\Conf$ is a finite-dimensional real vector space with a linear 
action of a compact group $G$. In particular, $\Conf\cong \Conf^*$ as
a
$G$-module. Let $\mu$ be a $G$-invariant measure
on $\Conf$, which always exists. We can find a growing $G$-invariant 
function $f(x)$ such that the map 
\begin{equation}
  \label{poly_to_functions}
   \Sd \Conf \otimes_\R \C \owns p(x) \mapsto p(x) e^{-f(x)} \in L^2(\Conf,\mu)
\end{equation}
has a dense image. Neither side of \eqref{poly_to_functions}
has a well-defined $G$-character 
because of infinite multiplicities, but the degree grading on
polynomials allows to form the following series 
$$
\sum_{k\ge 0}  t^k \, \tr_{\bS^k \Conf} g  = \exp\left( \sum_{n\ge 1} \frac{t^n}{n} 
\, \tr_{\Conf} g^n \right)\,, \quad g\in G\,,
$$
which will replace for us the $G$-character of $L^2(\Conf,\mu)$. 

\subsubsection{}

The odd degrees of freedom are
\begin{equation}
\Psi \oplus \Psi^* = \{ \textup{global sections of fermionic fields on
  $Z$}\}
\Big/ \sim \,,
\label{odd_conf}
\end{equation}
where fermionic fields of M-theory are the Rarita-Schwinger fields
$\psi_{\pm}$ of 
spin 3/2. They transform in the representations
$$
V^{3/2}_{\pm} = \Ker \left( V \otimes S_{\pm} \to S_{\mp} \right) 
$$
of the group $Spin(V)$ where 
$$
V=T_{z_0} Z \cong \R^{10}\,,
$$
and $S_\pm$ are the spinor representations of $Spin(V)$, the universal cover of $SO(V)$. 
The Lie algebra of the gauge transformations is also extended
to include the transformations 
$$
\psi \mapsto \psi + \nabla \textup{spinor}  \,, 
$$
that change $\psi$ by a derivative of a spinor field. If we linearize
around the vanishing RS fields then the configuration superspace 
is a direct product of its even and odd subspaces.

\subsubsection{} 

Building the space $L^2(\Conf,\Ld \Psi^*)$ requires a choice of the
polarization in \eqref{odd_conf}. An important point, which will 
be revisited below, is that the two natural choices
$$
\Psi^* = \{ \textup{global $\psi_\pm$ on $Z$} \} 
$$
give dual and \emph{inequivalent} answers. For now, we fix one 
choice, namely $\psi_+$. Then the $\Iso(g_0)$-index of $L^2(\Conf,\Ld
\Psi^*)$ is the symmetric algebra of global sections of the 
following virtual bundle 
\begin{alignat}{2}
\underline{\mathit{Conf}_\textup{super}} = & S^2 T Z - 1 &\qquad & \textup{traceless
  metric}  \label{Conf_bundle}\\
& - TZ &\qquad & \textup{modulo diffeomorphisms} \notag \\
 & \Omega^3 Z - \Omega^2 Z + \Omega^1 Z - 1 &\qquad & \textup{3-form
   modulo exact} \notag \\
& -TZ\otimes S_+ + S_+ + S_- &\qquad & \textup{RS field modulo exact}
\notag
\,. 
 \end{alignat}

\subsubsection{}

Now suppose that $g_0$ is a K\"ahler metric and choose $G\subset 
\Iso(g_0)$ so that it acts trivially on the trivial bundle
$\Omega^{5,0} Z$. Then 
$$
S_{\pm} = \Lambda^{\textup{even/odd}} T^{1,0} Z
$$
as $G$-bundles. A direct computation with characters proves
the following key 

\begin{Proposition}[\cite{Zth}] For $G$ as above, we have 
$$
\underline{\mathit{Conf}_\textup{super}} = - T^{1,0} 
\otimes(S_{+}-S_{-}) 
$$
as $G$-bundles and so, by Dolbeault, 
$$
L^2(\Conf, \Ld\Psi^*) = \Ld 
\chi (TZ)  \,, 
$$
as virtual $G$-modules, where $\chi$ is the holomorphic Euler characteristic. 
\end{Proposition}

\noindent 
Note that changing the roles of $S_+$ and $S_-$ changes the answer to 
$$
 \overline{\Ld 
\chi  (TZ)} = 
\Sd \chi (T^*Z)\,.
$$
The conjectural formula of \cite{Zth} for degree 0 DT invariants is, 
suitably interpreted, the product of \emph{both} answers, that 
is $\Sd \chi (T^*Z-TZ)$, see Conjecture \ref{c_DT0} below. 

It would be interesting to have a good explanation of this
doubling phenomenon, which may be compared to the squaring that happens
in the degree 0 part of the correspondence between cohomological 
Gromov-Witten and Donaldson-Thomas invariants of 3-folds \cite{mnop}. 
We don't discuss it further in the present paper and refer the 
reader to the original paper \cite{Zth} for more information. 

\subsubsection{}
Suppose $Z$ is the total space of two line bundles $\cL_1$ and 
$\cL_2$ over a 3-fold $X$, such that $\cL_1 \otimes \cL_2 = \cK_X$. 
 Then, by localization, 
$$
\chi(Z,TZ) = \chi
\left(X,\frac{TZ\big|_X}{(1-\cL_1^{-1}) (1-\cL_2^{-1})}\right)
$$
and further
\begin{multline}
  \frac{T^*Z\big|_X- TZ\big|_X}{(1-\cL_1^{-1}) (1-\cL_2^{-1})} =\cK_X
  - \cO_X + \\
  \frac{\cL_1}{(1-\cL_1) (1-\cL_2^{-1})}
\left( TX + \cK_X - T^*X- \cK_X^*\right) \label{deg0class1} \,. 
  \end{multline}
Now note that if $\Cq$ acts on $\cL_1$ and $\cL_2$ with weights 
$q$ and $q^{-1}$ respectively, then all weights occurring in the 
second line of \eqref{deg0class1} are positive, and therefore 
the symmetric algebra of that term is well-defined in
$\Cq$-equivariant
K-theory. Discounting the contribution of the 
first term in the RHS of  \eqref{deg0class1} 
as a (possibly infinite) prefactor, we make contact with 
the following reformulation of a conjecture from \cite{Zth}. 

\subsubsection{}
Let 
$$
\Hilb(X,\textup{points}) = \bigsqcup_{n\ge 0} \Hilb(X,n) 
$$
be the Hilbert scheme of points of $X$ and consider
the following sheaf on it
$$
\tO_\textup{vir} = (-q)^{\chi(\cO/\cI)} \, \Ovir \otimes 
\left(\Kvir \otimes \det
H^0(\cO/\cI \otimes(\cL_1-\cL_2))\right)^{1/2}
$$
where the $\cO/\cI$ is the structure sheaf of the universal
$0$-dimensional subscheme of $X$. This is special case of the 
sheaf $\tO_\textup{vir}$ defined and discussed below, so we don't 
go into a further discussion of it here.

\begin{Conjecture}[\cite{Zth}]\label{c_DT0} For $X$ as above 
$$
\chi\left(\Hilb(X,\textup{points}),\tO_\textup{vir}\right) = 
\Sd \chi\left(X,\frac{q \cL_1 \left( TX + \cK_X - T^*X- \cK_X^*\right)}{(1-q\cL_1) (1-q\cL_2^{-1})}
\right)
$$
\end{Conjecture}

We expect that the conjectural PT/DT correspondence
\cite{PT} extends to K-theoretic invariants as follows 
\begin{equation}
  \label{ePTDT}
  \chi\left(\PT(X),\tO_\textup{vir}\right)  \overset?= 
\frac{\chi\left(\Hilb(X,\textup{curves}),\tO_\textup{vir}\right)}
{\chi\left(\Hilb(X,\textup{points}),\tO_\textup{vir}\right)} \,. 
\end{equation}

\subsubsection{}
To get a sense what this means in concrete 
terms, take $Z=\C^5$ and let $t_1,\dots,t_5$ be the
weights of the coordinate directions. They satisfy 
\begin{equation}
  \label{prodt}
\prod t_i =1 \,.
\end{equation}
We have 
$$
\chi(Z,TZ^*) = \frac{\sum_1^5 t_i^{-1}}{\prod_1^5 (1-t_i^{-1})} \,,
$$
where the denominator may be symmetrized using \eqref{prodt}. 
Identity \eqref{deg0class1} says that 
\begin{multline}
\frac{\sum_1^5 t_i^{-1}-\sum_1^5 t_i}{\prod_1^5 (t_i^{1/2}-t_i^{-1/2})}
= \\ 
\frac{-(t_1 t_2 t_3)^{1/2}+(t_1 t_2 t_3)^{-1/2}}
{\prod_1^3 (t_i^{1/2}-t_i^{-1/2})}  \, + 
\frac{\prod_{i<j\le 3}  \left((t_i t_j)^{1/2}-(t_i t_j)^{-1/2}\right)}
{\prod_1^5 (t_i^{1/2}-t_i^{-1/2})}\label{id0cl}
\end{multline}
whenever \eqref{prodt} is satisfied.

 The first term in the left-hand
side of \eqref{id0cl} does 
not involve $t_4$ or $t_5$ and may be viewed as a perturbative
contribution to the integrals over the Hilbert schemes. 
The second term in the left-hand side of \eqref{id0cl}, which we
will denote $F(t)$, and in which one can already recognize the 
contribution of the Hilbert scheme of 1 point, computes the 
degree zero DT invariants as follows. 

The bundles $\cL_1$ and $\cL_2$ are trivial bundles with weights $t_4$
and 
$t_5$, respectively. Hence on the Hilbert scheme $\Hilb(\C^3,n)$
of $n$ points the line bundle 
$\det H^0(\cO/\cI \otimes(\cL_1-\cL_2))$ is trivial with weight $(t_4/t_5)^n=q^{2n}$. 
Therefore, Conjecture \ref{c_DT0} means 
$$
\sum_{n\ge 0} (-q)^n \chi\left(\Hilb(\C^3,n), \Ovir \otimes 
\Kvir^{1/2}\right) \overset?= 
\exp \left( \sum_{n=1}^\infty \frac{F(t^n)}{n} \right)
$$
with 
$$
t_4 = q^{1/2} (t_1 t_2 t_3)^{-1/2} \,, \quad 
t_5 = q^{-1/2} (t_1 t_2 t_3)^{-1/2} \,. 
$$
The fact that the right-hand side of \eqref{id0cl} has a full
5-dimensional symmetry is a very nontrivial confirmation of the 
M-theory paradigm.

\section{The DT integrand}\label{s_DT} 

\subsection{The modified virtual structure sheaf} 

\subsubsection{}

The DT moduli spaces of $X$, in their original definition \cite{Thom}, 
parameterize ideal sheaves $\cI \subset \cO_X$ of $1$-dimensional 
subschemes in $X$. They have a perfect obstruction theory described by 
\begin{align}
\Def - \Obs &= \chi(\cO_X) - \chi(\cI,\cI) \label{defI}
\\
 &= \chi(\cF) + \chi(\cF,\cO_X) - \chi(\cF,\cF) \notag\,,
\end{align}
where $\cF = \cO/\cI$ is the  universal $1$-dimensional sheaf on $X$. 

Other stability conditions for complexes of sheaves on $X$ lead to alternative 
DT moduli spaces. In particular, in the Pandharipande-Thomas chamber \cite{PT}, 
the moduli spaces parameterize pairs
$$
s: \cO_X \to \cF 
$$
where $\cF$ is a pure $1$-dimensional sheaf and the cokernel of the 
section $s$ has finite length. The formula for the $K$-theory class of
their obstruction theory is the same. 

\subsubsection{}

A perfect obstruction theory defines, in the usual way (see 
for example   
\cite{FG}), 
a virtual structure sheaf $\Ovir$. If, for instance,
 $p\in \PT(X)$ is an isolated
fixed point of a torus $\bT$ and 
$$
\Def_p - \Obs_p = \sum a_i - \sum b_i \,, \quad a_i,b_j \in \bT^\vee\,, 
$$
is the character of the deformation theory at $p$ then
the localization of $\Ovir$ at $p$ in $\bT$-equivariant 
$K$-theory equals
\begin{equation}
\cO_{\textup{vir},\textup{localized}} = 
\frac{\prod \left(1-b_i^{-1}\right)}
{\prod \left(1-a_i^{-1}\right)}  \,. 
\label{locOvir}
\end{equation}

\subsubsection{}

The 
virtual canonical bundle is defined by 
$$
\Kvir = \frac{\det \Obs}{\det \Def} \,. 
$$
Of particular importance to us will be its square root 
$\Kvir^{1/2}$, compare with Section \ref{s_Spinor}.

Different
choice of the square roots (related by the $2$-torsion in 
the Picard group) correspond to different boundary conditions 
for fermions in the theory. This means they define different 
sectors of the theory that have to be matched in concrete
computations. 

If $p$ is an isolated fixed point as in \eqref{locOvir} 
then 
\begin{equation}
\left(\Ovir \otimes \Kvir^{1/2} \right)_\textup{localized} = 
\frac{\prod \left(b_i^{1/2}-b_i^{-1/2}
\right)}{\prod \left(a_i^{1/2}-a_i^{-1/2}\right)}  \,. 
\label{locKvir}
\end{equation}

\subsubsection{}\label{s_Hodge} 
The twist by $\Kvir^{1/2}$ brings K-theoretic DT computations 
much closer to familiar sheaf cohomology problems. 

There is 
a certain degree of duality between deformations and obstructions 
in DT theory, with perfect duality in the case when $\cK_X$ 
restricts to the trivial bundle on the support of $\cF$. 
It is, therefore, useful 
to keep in mind the following baby example of a self-dual 
obstruction theory, which we will revisit below. 

Let $M$ be a smooth algebraic variety, viewed as zero section 
$$
s: M \to T^*M
$$
of its cotangent bundle. The corresponding obstruction theory is 
$$
\Def - \Obs = TM - T^*M
$$
with 
$$
\Ovir = s^*(\cO_M) = \sum_{k=0}^{\dim M} (-1)^k \Lambda^k TM \,.
$$
We have 
$$
\Kvir^{1/2} = \cK_M 
$$
and hence
\begin{equation}
\Ovir \otimes \Kvir^{1/2} = (-1)^{\dim M} 
\sum_{k=0}^{\dim M} (-1)^k \, \Omega^k M \,.
\label{OvirOmega}
\end{equation}%
If a torus scales the fibers of $T^*M$ with weight $t$ then 
it scales the $\Omega^k M$ term in \eqref{OvirOmega} with weight 
$t^{k-\dim M/2}$. Thus 
$$
\chi \left(\Ovir \otimes \Kvir^{1/2}\right) = 
\sum_{p,q} (-1)^{p-\dim/2} (-t)^{q-\dim/2} H^p(\Omega^q M)  
$$
is a specialization of the centered Hodge polynomial of $M$. 

\subsubsection{}

A sheaf $\cG$ on $X$ gives a line bundle 
$$
\bL_\cG 
= \det H^*(X,\cF \otimes \cG)
$$ 
on the DT moduli spaces, where 
$\cF$ is the universal $1$-dimensional sheaf, e.g. $\cF = \cO/\cI$
in the Hilbert scheme chamber. 

\subsubsection{} 

If $\cG$ is $1$-dimensional, then the degree of $\bL_\cG$ 
may be computed as follows. Let $B$ be $1$-dimensional family 
of sheaves and let $\cF$ denote  the corresponding sheaf on $B \times X$. 
Let 
$$
p_X: B \times X \to X 
$$
denote the projection and let 
$$
\sweep(B) = p_X \left(\cycle(\cF)\right)
$$
denote the 2-cycle in $X$ swept by the cycles of sheaves in $B$. 
From Grothendieck-Riemann-Roch, we have 
\begin{equation}
\deg_{B} \bL_\cG = \sweep(B) \, \cdot \, \cycle(\cG) \,. 
\label{degbL}
\end{equation}
Note that families with trivial sweep are precisely those 
contracted by the map to the Chow variety and, in fact, it 
can be shown \cite{JO} that the bundle $\bL_\cG$ is pulled back 
from the Chow variety. 

\subsubsection{} 

Symmetrically, for 1-dimensional $\cG$, 
the bundle $\bL_\cG\in \Pic \DT(X)$ depends
only on $\cycle(\cG)$ and, in fact, only on its rational 
equivalence class. 








\subsubsection{}

We now have prepared all ingredients for the definition of the 
modified virtual structure sheaf $\tO_{\DT}$. 

Let $X$ be the fixed locus of $\Cq$-action on a nonsingular Calabi-Yau 
5-fold $Z$ 
as above. Since $\Cq$ preserves the holomorphic 5-form $\Omega^5$, we
have 
$$
N_{X} Z = \cL_1 \oplus \cL_2
$$
where $\Cq$ acts with weights $q$ and $q^{-1}$, respectively. The 
triviality of $\cK_Z$ implies
$$
\cL_1 \otimes \cL_2 = \cK_X \,.
$$
The roles 
of line bundles $\cL_1$ and $\cL_2$ will \emph{not} be symmetric, reflecting 
the choice stressed in Section \ref{coneCq}: the $\cL_1$ direction 
is attracting as $q\to 0$, while the $\cL_2$ direction is repelling.

Given a $1$-dimensional sheaf $\cF$ on $X$, we denote by 
$$
\chi = \chi(\cF)\,, \quad \beta = \left[\cycle(\cF)\right] \in H_2(X,\Z)
$$
its discrete invariants. The virtual dimension of DT moduli spaces at 
a point corresponding to $\cF$ equals 
$$
\textup{vir dim} = - (\cK_X,\beta) = - (\cL_1+\cL_2,\beta) \,. 
$$
\begin{Definition}
We define 
\begin{equation}
  \label{deftODT}
  \tO_{\DT} = \textup{prefactor}  \, \, 
 \Ovir \otimes \left( \Kvir \otimes \bL_{\cL_1-\cL_2} \right)^{1/2}
\end{equation}
where 
\begin{align}
  \label{deftODT2}
  \textup{prefactor} &= 
(-1)^{\frac{(\cL_1-\cL_2,\beta)}{2}} \, (-q)^{-\frac{\textup{vir dim}}2+\chi} \, 
Q^\beta \\
&= 
(-1)^{(\cL_1,\beta)+\chi} \, q^{-\frac{\textup{vir dim}}2+\chi} \, 
Q^\beta \notag \,. 
\end{align}
\end{Definition}
Note that in the $\bL_{\cL_1-\cL_2}$ term we have the difference of 
$K$-classes and not the ratio $\cL_1 \otimes \cL_2^{-1} \in \Pic(X)$. 
There is a simple explanation for the this form of the DT 
integrand, see Section \ref{s_lowest_term}.

We will see in Section \ref{pull_square} that 
$$
\Kvir \otimes \bL_{\cL_1-\cL_2} = \textup{square} \otimes 
\bL_{c_1(\cL_1) \cap 
c_1(\cL_2)}\,, 
$$
where the second factor is pulled back from the Chow variety of $X$, 
as in the preceding discussion.  

\subsection{The interaction term $\bPhi$}

\subsubsection{$A_{r-1}$ surface fibrations}\label{sAr} 

For the discussion of the interaction between different components of 
$Z^{\Cq}$ it is convenient to keep in mind the following simplest 
example. Let 
$$
Z_1 = 
\begin{matrix}
\cL_1 \oplus \cL_2  \\
\downarrow\\
X
\end{matrix} \,, \quad \cL_1 \otimes \cL_2 = \cK_X \,,
$$
be the total space of two line bundles over $X$. Let 
$\Cq$ scale the fibers by $\textup{diag}(q,q^{-1})$ and 
let $\mu_r\subset \Cq$ be the group of $r$th roots of unity. 
Let 
$$
\xymatrix{
A_{r-1} \ar@{^{(}->}[r] & Z_r  \ar@{->}[d]\\
& X 
}
$$
be the minimal resolution of the quotient $Z_1\big/\mu_r$. It 
fibers over $X$ in $A_{r-1}$-surfaces
$$
A_{r-1} = \widetilde{\C^2/\mu_r}\,,
$$
that is, minimal resolutions of the the singularity $x^{r} = yz$. 
The quotient in 
$$
1 \to \mu_r \to \Cq \xrightarrow{\,\,\,
 q\mapsto q^r \, } \Ct_{q^r} \to 1 
$$
acts canonically on $Z_r$ and
$$
Z_r^{\Ct_{q^r}} = X \times \{ \textup{$r$ points}\}  \,. 
$$
In this example, we will see the rank $r$ Donaldson-Thomas
theory on $X$ appear from the interaction of rank $1$ 
theories on $r$ copies of $X$, see Section \ref{eng_Hr}. 

\subsubsection{Unbroken curves}

Going back to the general situation, let $C\subset Z$ be a reduced
connected 
$\Cq$-invariant curve. We say that $C$ is \emph{unbroken} if 
$\Cq$ acts nontrivially
on each component of $C$. This implies $C$ is rational, at worst nodal, 
and that the 
two branches at each node have opposite weights. It also implies it has
two nonsingular fixed points $p_1,p_2 \in C$ which lie on two different 
components $X_1$ and $X_2$ of the fixed locus. We say that $C$ \emph{flows
from $X_2$ to $X_1$} if the $\Cq$-weight of $T_{p_1} C$ is
positive. We denote by $\cU_{21}$ the moduli space of unbroken curves 
from $X_2$ to $X_1$. 

Since both spaces $T_{p_i} Z$ have three trivial $\Cq$-weights, there 
are two possibilities for the normal bundle to $C$, namely 
\begin{equation}
N_C Z = 
\begin{cases}
\cO(-p_1-p_2) \oplus \cO^{\oplus 3}\,, &\textup{or}\\
\cO(-p_1) \oplus \cO(-p_2) \oplus \cO^{\oplus 2} \,, 
\end{cases}\label{NCcases}
\end{equation}
as $\Cq$-equivariant sheaves. 
In the first case in \eqref{NCcases},
$$
\cU_{21} \cong X_1 \cong X_2 
$$
with the
obstruction bundle 
$$
H^1(C,\cO(-p_1-p_2)) \cong \cL_1\big|_{p_i} \otimes \cL_2\big|_{p_i} 
\cong 
\cK_{X_i} \big|_{p_i} \quad i=1,2 \,,
$$
where, as usual, we denote by $\cL_1$ and $\cL_2$ the $\Cq$-eigenbundles
in the normal bundle to the fixed locus. In the second
 case in \eqref{NCcases}, 
the deformations
are $2$-dimensional and unobstructed. The moduli space of unbroken 
curves embeds 
$$
X_2 \hookleftarrow \cU_{21} \hookrightarrow X_1 
$$
in each of the $X_i$ as a smooth surface. 

\subsubsection{Threefold and surface interactions}\label{s_23inter}

We will refer to the two cases in \eqref{NCcases} as the threefold and 
surface interactions, respectively. For example, there is a threefold 
interaction between any pair of fixed components in the 
example of Section \ref{sAr}. The corresponding unbroken curves 
are the $(-2)$-curves in the $A_{r-1}$-fibers. 

An example of a surface interaction may be constructed as follows. 
Take 
$$
Y = 
\begin{matrix}
\cO(-1)\oplus \cO(-1) \\
\downarrow\\
\pP^1
\end{matrix} 
$$
and make $\Cq$ act on $Y$ by scaling the base $\pP^1$ and so that
there is a trivial weight in the fiber over each fixed point. 
Instead of $Y$ we could have taken many other toric CY threefold
that contain the $\Cq$-invariant curve with an $\cO(-1)\oplus \cO(-1)$ 
normal bundle. Consider a $Y$-bundle over a surface $S$ associated to 
a principal $(\Ct)^3$-bundle $P$
$$
Z = P \times_{(\Ct)^3} Y\,, \quad c_1(P) = \cK_S
$$
We see that $\cK_Z \cong \cO_Z$, $X_1$ and $X_2$ are line
bundles over $S$,  and 
$$
\cU_{21} \cong S
$$
is embedded in each of them as the zero section. Since the surface $S$
is arbitrary, we conclude $\cK_{\cU_{21}}$ may not be 
a square.

\subsubsection{The operators $\Phi_{ij}$}

The obstruction theory 
$$
\Def - \Obs = H^*(C,N_C Z)
$$
gives $\cU_{21}$ a virtual structure sheaf of virtual dimension $2$. 
For 3-fold interactions, the corresponding virtual
canonical bundle is 
always a square with 
$$
\tO_{\cU_{21}} = \cK_{X_i}- \cO_{X_i} \,. 
$$ 
For surface interactions, we \emph{assume} that the square root in 
$$
\tO_{\cU_{21}}  = \cK^{1/2}_{\cU_{21}} 
$$
exists and we define, in either case, 
$$
\Phi_{21} = \eval_*  \tO_{\cU_{21}} \,, 
$$
where 
$$
\eval: \cU_{21} \to X_2 \times X_1
$$
sends an unbroken curve $C$ to the fixed points $(p_2,p_1)$. 
We will denote by the same symbol $\Phi_{21}$ the corresponding 
Fourier-Mukai operator 
$$
\cF_1 \xrightarrow{\,\,\Phi_{21}\,\,} p_{X_2,*} 
\left(\Phi_{21} \otimes
p_{X_1}^* \, \cF_1 \right) \,.
$$
For example, for 3-fold interactions 
$$
\Phi_{21} \, \cF_1 = \cK_X \otimes \cF_1 - \cF_1 \,. 
$$

\subsubsection{The interaction}\label{s_inter} 

Let $X_1$ and $X_2$ be two components of $X$ as above 
and let $\cF_1$ and 
$\cF_2$ denote the universal $1$-dimensional sheaves over the 
DT moduli spaces for $X_1$ and $X_2$. Using the operator $\Phi_{21}$, 
we can define the following $K$-theory class
\begin{equation}
\cX_{21} = \chi(\cO_{X_2}-\cF_2, \Phi_{21} (\cO_{X_1}-\cF_1)) \label{defX21}
\end{equation}
which may be compared to the formula \eqref{defI} for the 
virtual tangent space to DT moduli spaces. 

We define
\begin{equation}
\bPhi = \bigotimes_{i<j} \bS \, Q^{[C_{ij}]} \, \cX_{ji}\,,
\label{defbPhi}
\end{equation}
where $S$ denotes the symmetric algebra, 
$$
[C_{ij}] \in H_2(Z,\Z)
$$
is the class of the unbroken curve flowing from $X_j$ to $X_i$, 
and the indexing is the components is such that curves flow 
from larger components to smaller ones. 

\subsubsection{An example}

For 3-fold interactions, we have, using Serre duality 
\begin{equation}
  \cX_{21} =- 
\chi(\cI_2,\cI_1) - \overline{\chi(\cI_1,\cI_2)} \label{XA}  \\
\end{equation}
where $[\cI]=[\cO_X] - [\cF]$ and bar denotes dual. Such 
interaction terms occur naturally in higher rank DT theory, see 
Section \ref{e_rank} below. 

\subsubsection{Perturbative contributions}\label{s_pert} 

Note that the Euler characteristic \eqref{defX21} may not be
well-defined if $X_i$'s are not proper. However, the difference 
\begin{multline}
 \cX'_{21}= \cX_{21} - \chi(\cO_{X_2},\Phi_{21} \cO_{X_1}) =\\
 -\chi(\cF_2,\Phi_{21} \cO_{X_1})
- \chi(\cO_{X_2},  \Phi_{21}  \cF_1)
 +\chi(\cF_2, \Phi_{21} \cF_1)
\label{X21red}
\end{multline}
is well-defined and differs from \eqref{defX21} only by a constant, 
even if infinite-dimensional, vector space $\chi(\cO_{X_2},\Phi_{21}
\cO_{X_1})$. The character of its symmetric algebra 
$$
\bS \, Q^{[C_{12}]} \, \chi(\cO_{X_2},\Phi_{21}
\cO_{X_1})
$$
may be regularized using any of the traditional
approaches. From the point of view of DT theory on X, it comes 
out as an overall prefactor, also known as a perturbative
contribution.

\section{The index of membranes}\label{s_M} 

\subsection{Membrane moduli}

\subsubsection{Multiple curves} 

Recall that the moduli space $\MM(Z)$ of stable membranes in 
$Z$ is supposed to be a certain compactification of the moduli
space of immersed holomorphic curves $C\subset Z$. One such compactification 
is the moduli space of stable maps; compactifications using moduli
of sheaves on $Z$ may also be considered.  While it is entirely 
possible that the M2-brane contributions to the M-theory indexed may 
be calculated using such moduli spaces, in this paper we pursue 
an alternative route. 

The main geometric difficulty in dealing with holomorphic curves
is degeneration to multiple curves, e.g. the ellipse
$$
y^2 = \lambda (1-x^2)
$$
degenerating to the double line $y^2=0$ as $\lambda\to 0$. A physicist
may call a multiple curve a bound state  
of several M2-branes. In the moduli space of stable maps, the $\lambda\to 0$
limit is the double cover of the $y=0$ line branched over
the points $x=\pm 1$, which remember the branchpoints of the $x$-projection 
of the original ellipse. In the Hilbert scheme of curves, the limit would 
just be the subscheme of the plane cut out by $y^2=0$, with 
no memory of the shape of the original conic. 

One reason we don't try to construct membrane moduli using sheaves
on $Z$ or stable maps to $Z$ is that these don't give natural bounded
moduli spaces for given degree, recall the discussion of Section
\ref{Mbound}. For example, in the above example, there could be 
double covers of $y=0$ with an arbitrary large number of
branchpoints or this line may be the support of a sheaf with an arbitrary large
Euler characteristic.

\subsubsection{Maps from schemes}\label{s_mapsf} 

In this paper, we look at maps
\begin{equation}
f: C \to Z
\label{fCZ}
\end{equation}%
from $1$-dimensional \emph{schemes} $C$ to $Z$. In the above example, 
this would be just the inclusion of the double line. In general 
$f$ need not be injective, like in the case of an immersion of a 
smooth curve $C$. 

In practical terms, a map $f$ may be represented
by a subscheme
$$ 
C \subset Z \times \pP^N
$$
for some $N\gg 0$, with the map $f$ being the projection to the first 
factor. Using such presentation, one defines the normal sheaf to 
the map $f$ by 
$$
N_f = N_C \left(Z \times \pP^N\right) - \cO_C \otimes T \pP^N \,.
$$
Here 
$$
N_C \left(Z \times \pP^N\right) = \Hom (\cI_C, \cO_C)\,,
$$
where $\cI_C$ is the ideal sheaf of $C$ and $\cO_C$ is its structure
sheaf. 

When $C$ is nice, e.g.\ smooth or a local complete intersection, 
$N_f$ is a vector bundle of rank 
$4$ and degree $2g-2$. However, in general it can be much larger, 
reflecting the singularities of the moduli spaces of maps 
\eqref{fCZ}. Some strategies for dealing with large $N_f$ 
will be discussed below.

\subsubsection{Stability conditions}\label{s_stabil} 

We impose the following \emph{stability} conditions on the maps \eqref{fCZ}: 
\begin{enumerate}
\item[(1)] The map $f$ is an isomorphism on its image 
away from a finite set of points in $C$. 
\item[(2)] For any proper subscheme 
$C' \subset C$
\begin{equation}
\frac{\chi(\cO_{\C'})}{\deg f(C')} >  
\frac{\chi(\cO_{\C})}{\deg f(C)}
\label{slopestab}
\end{equation}
which means that $\chi(\cO_{\C'}) f(C) - \chi(\cO_{\C}) f(C')$ 
is a nonnegative and nonzero linear combination of the components of
$\supp f(C)$. 
\end{enumerate}

\noindent 
For example, a double line $C\subset \pP^2$ is stable since 
$\chi(\cO_{\C}) = 1$ and $\chi(\cO_{\C'}) \ge 1$ for any subscheme of $C$.
 
More generally, let $C$ be double zero section inside the total space 
of line bundle $\cL$ over a curve $B$. Then 
$$
\chi(\cO_C) = 2 \chi(\cO_B) - \deg \cL 
$$
and so \eqref{slopestab} means 
$$
\textup{$C$ is stable} \,\, \Leftrightarrow \,\, \deg \cL > 0 \,. 
$$
In other words, membranes can only stack up in \emph{positive direction} 
of the normal bundle. 

\subsubsection{CM  property} 

A $1$-dimensional scheme is Cohen-Macaulay if for every point $x\in C$ 
there is a function $f$ vanishing at $x$ which is not a zero-divisor.
 
If this condition is violated at some point $x\in C$ then  
$$
\cI_{C'}=\Ann \fm_x \subset \cO_C\,,
$$
where $\fm_x$ is the ideal of functions vanishing at $x$, is a 
nontrivial ideal of finite length. Thus
$$
\cO_{C'} = \cO_{C} / \cI_{C'} 
$$
is a proper subscheme with 
$$
\chi(\cO_{C'}) < \chi(\cO_{C})\,, \quad [f(C')] =  [f(C)] \,. 
$$
Therefore, the sources $C$ of all stable maps \eqref{fCZ} are 
Cohen-Macaulay. 

Maps from 1-dimensional 
Cohen-Macaulay schemes to projective varieties
were studied by H{\o}nsen 
in \cite{Hon}, who constructed their moduli space as a proper separated
algebraic space for given $\deg f(C)$ and $\chi (\cO_C)$. 
He imposes the first, but not the second stability condition 
in Section \ref{s_stabil}. 

\subsubsection{Boundedness}

For any  map \eqref{fCZ}, the Euler characteristic $\chi(\cO_C)$ 
may be bounded \emph{below} in terms of the degree of $f(C)$. 
Stability \eqref{slopestab} also bounds it from \emph{above}. 
Therefore, stable maps \eqref{fCZ} form a bounded 
family once the degree of the map is fixed. 

This is natural from the M-theory perspective. In M-theory there 
is a 3-form which couples to the worldvolume 
$C \times S^1$ of the membrane and thus keeps track of its degree. 
On the other hand, there are no fields that couple to the Euler
characteristic of $C$ and, besides, the Euler characteristic 
of  $C \times S^1$ vanishes, as it does for any smooth real 3-fold. 

This means that on the membrane side of our conjectures, we 
sum over all Euler characteristics of membranes with \emph{no weight}. 
The $q^\chi$-weight on the DT side appears only 
because of the $\Cq$-action, the existence of which is an additional 
hypothesis on $Z$. 

\subsection{Deformations of membranes} 

\subsubsection{}

When the normal sheaf $N_f$ becomes too big, the deformation 
theory of a map \eqref{fCZ} becomes very complicated and technical. 
Perhaps some form of a virtual structure sheaf may be constructed
from the normal complex of $f$. At this time, however, we are planning 
to pursue a more geometric approach, namely to take as $\MM(Z)$ a 
certain \emph{virtual Nash blowup} of the H{\o}nsen space. 

Recall that the ordinary Nash blowup of a singular space $\cM$
remembers the limits of tangent spaces at the smooth points 
$\cM_0 \subset \cM$ as they approach the singularities. A point 
of Nash blowup of $\cM$ is described by a pair $(p,N)$, where 
$p\in \cM$ and 
$$
N \subset T_p \cM \,, \quad \dim N = \dim \cM_0 \,.
$$

\subsubsection{}

Our hypothetical virtual blowup $\MM(Z)$ 
of the H{\o}nsen space should parameterize 
maps \eqref{fCZ} together with a subsheaf 
\begin{equation}
N \subset N_f\label{NinNf}
\end{equation}
of class 
$$
[N] = 3 \left[\cO_C\right] + \left[\omega_C \right] 
$$
in $K$-theory of $C$, where $\omega_C$ is the dualizing sheaf of a 
Cohen-Macaulay scheme $C$. 
Additional conditions on $N$ form a 
subject of current research and will be discussed separately. 

A possible physical interpretation of the extra data 
contained in \eqref{NinNf} is the following. The map 
$f: C\to Z$ is really the bosonic part of a map of 
\emph{superschemes}, the fermionic part of which 
is uniquely reconstructed in the case when $f$ is an 
immersion or a more general l.c.i.\ map. The 
uniqueness of the reconstruction fails when $f$ develops
singularities and the subsheaf \eqref{NinNf} stores the 
missing information. 

\subsubsection{}

With these additional conditions, we hope $\MM(Z)$ to have 
an obstruction theory with 
$$
\left[ \Def - \Obs \right] = \left[H^*(N)\right]
$$
in $K$-theory of $\MM(Z)$. We don't expect these virtual bundles
to be isomorphic, it is only their pieces with respect to a certain
filtrations that should be identified.

\section{Examples}
\label{s_exa}

\subsection{Reduced local curves}\label{s_slc} 

\subsubsection{}

Let $Z$ be the total space of 4 line bundles 
$$
Z = 
\begin{matrix}
\cL_1 \oplus \cL_2 \oplus \cL_3 \oplus \cL_4 \\
\downarrow\\
C
\end{matrix} \,, \quad \bigotimes \cL_i = \cK_B \,,
$$
over a smooth curve $C$. As before, we make $\Cq$ act on $\cL_1$ and 
$\cL_2$ with weights $q$ and $q^{-1}$ and hence $X$ is the total space
of $\cL_3\oplus \cL_4$. We want to compare the DT and M-theoretic 
counts for the zero section $C$ inside $Z$. 

\subsubsection{}

A 3-dimensional torus
$\bT$ acts on $Z$ scaling the individual $\cL_i$'s. Clearly, 
$$
\MM(Z,[C])^\bT = \{0\}
$$
is a point representing the curve $C$ itself. We have 
$$
\left(\Def^{\MM}-\Obs^{\MM} \right)\Big|_{0} = H^*(N_C Z) \,, \quad N_C Z = 
\bigoplus_{i=1}^4 \cL_i \,. 
$$
Therefore 
$$
\chi\left(\MM,\tO_\textup{vir}\right) = \left(\det H
\right)^{-1/2}  \Sd  H^\vee \,, \quad H=H^*(N_C Z)\,. 
$$
In practice, this means that if $\sum a_i - \sum b_i$ is the character 
of $H^*(N_C Z)$ then 
\begin{equation}
\chi\left(\MM,\tO_\textup{vir}\right) = 
\prod \frac{b_i^{1/2} -
  b_i^{-1/2} }
{a_i^{1/2} -
  a_i^{-1/2} } \,. \label{prodm}
\end{equation}

\subsubsection{}

For comparison with DT theory, we need to expand \eqref{prodm} in 
powers of $q$. It is convenient to separate the $\Cq$-moving directions 
$$
N_{12} = \cL_i \oplus \cL_j\,, \quad H_{12} = H^*(N_{12}) \, ,
$$
and their contribution 
$$
\chi(\MM,\tO_\textup{vir})_{12} = \textup{contribution of
  $N_{12}$}
$$
to \eqref{prodm}. 
We compute 
\begin{multline}
\chi(\MM,\tO_\textup{vir})_{12}= 
(-1)^{h_1} q^{\frac{h_1 + h_2}2} \otimes \\ 
(\det H^*(\cL_1-\cL_2))^{1/2} \,\,  \Sd \left(q H^*(\cL_1) \oplus
q H^*(\cL_2)^\vee\right) \,, 
 \label{Minq} 
\end{multline}
where 
$$
h_i = \rk H^*(\cL_i) = \deg \cL_i + 1 - g(C) \,.
$$
In particular 
\begin{equation}
\chi(\MM,\tO_\textup{vir})_{12} = 
(-1)^{h_1} q^{\frac{h_1 + h_2}2} \, 
(\det H^*(\cL_1-\cL_2))^{1/2} \, (1+O(q)) \,,
 \label{Minql} 
\end{equation}
as $q\to 0$. 

\subsubsection{} 

On the Donaldson-Thomas side, we have 
$$
\PT(X,[C])^\bT = \bigsqcup_{n\ge 0} S^n C \,. 
$$
The deformation theory of these spaces consists of  deforming 
$C$ in the $N_{34}$-direction and a certain twisted cotangent 
bundle on $S^nC$, see below. In particular, the contribution of 
$H_{34}$ to PT counts is precisely $\chi(\MM,\tO_\textup{vir})_{34}$. 

Recall that by definition \eqref{deftODT}
$$
\tO_{\DT}  = \textup{prefactor}  \, \, 
 \Ovir \otimes \left( \Kvir \otimes \det H^*(\cL_1-\cL_2)
 \right)^{1/2}
$$
where, dropping the constant $Q^C$ term, 
\begin{align}
  \textup{prefactor}
&= 
(-1)^{(\cL_1,\beta)+\chi} \, q^{-\frac{\textup{vir dim}}2+\chi} \, 
\notag \\
&= 
(-1)^{h_1+n} \, q^{\frac{h_1+h_2}2+n} \, 
\label{PTinql}
\end{align}
because $\chi(\cF)=n+1-g(C)$ for sheaves $\cF$ parameterized by $S^n C$ 
and 
$$
- \textup{vir dim} = \deg \cK_X = \deg \cL_1 + \deg \cL_2  \,. 
$$

\subsubsection{}\label{s_lowest_term} 

The lowest term in the $q$-expansion corresponds to $n=0$. Comparing 
\eqref{Minql} to \eqref{PTinql} we find a perfect agreement. In fact,
  we see that the form of the prefactor \eqref{deftODT2} is dictated
by the lowest $q$-term for reduced local curves. 

\subsubsection{}

For $n>0$ there is a nontrivial obstruction bundle on $S^n C$. 
When $\cK_X$ is trivial, that is, when 
$$
\cL_3 \otimes \cL_4 = \cK_C 
$$
this is the cotangent bundle to $S^n C$ by the duality between 
deformations and obstructions. In general, it is a certain twisted
version of $T^* S^n C$. 

Let 
$$
\Delta \subset S^n C \times C
$$
be the universal subscheme. Recall \cite{ACGH} that 
\begin{equation}
T^* S^n C = (p_1)_* \, \cO_\Delta \otimes p_2^*(\cK_C)\,,\label{TSn}
\end{equation}
where $p_i$ are the projections to the two factors.
More generally, 
$$
\Obs =   (p_1)_* \, \cO_\Delta \otimes p_2^*(\cL_3 \otimes \cL_4 )\,. 
$$
We note that the formula
$$
\boxed{\chi\left( \oSd C, \tO_{\DT} \right) = \left(\det H
\right)^{-1/2}  \Sd  H^\vee\,,} \qquad H=H^*(\oplus \cL_i )\,, 
$$
is a generalization of a classical formula of Macdonald for the 
Hodge numbers of symmetric powers of a curve. Presumably, 
it has an elementary proof.

\subsection{Double curves}

\subsubsection{}
Let $\cL$ be a line bundle on a smooth curve $B$ and let 
$S_\cL$ be 
the total space of this line bundle. If $z$ is the local coordinate
along the fibers of $S_\cL$ then 
$$
\bB_\cL = \{z^2=0\} \subset S_\cL
$$
is the infinitesimal thickening of the base $B$ in the fiber
direction. We have 
$$
\cO_{\bB_\cL} = \cO_{B} \oplus \cL^{-1}
$$
as $\cO_B$-module and, in particular,
$$
\chi(\bB_\cL) = 2 \chi(B) - \deg \cL \,.
$$
The normal bundle to $\bB_\cL$
$$
N_{\bB_\cL} S_\cL = \cL \oplus \cL^2 
$$
may be seen concretely as deformations of the form 
$$
\{ z^2 + p_1 z + p_2 = 0\}\,, \quad p_i \in \Gamma(\cL^i) \,.
$$
A very familiar example, in which there is no $H^1$ of the normal 
bundle, 
is the deformations of the double line to a conic in $\pP^2$. 

\subsubsection{}
Let $D\subset B$ is an effective divisor of degree $d$ and let
$$
s_D \in H^0(B,\cO(D))
$$
be the tautological section. It defines a map 
$$
F_D: S_{\cL(-D)} \owns (b,z) \mapsto (b,s_D(b) z) \in S_\cL 
$$
where $(b,z)$ are the base and the fiber coordinates in the domain 
of $F_D$.  

The map $F_D$ is the blowup of $S_\cL$ in the subscheme $D\subset
S_\cL$. Its deformations have the form 
$$
\Def(F_D)= T_{D} \Hilb(S_{\cL},d) = T_D \, S^d B + H^0(B,\cO_D
\otimes \cL)
$$
and they are unobstructed. We already saw the tangent space 
to the symmetric power $S^dB$ of a curve $B$ in \eqref{TSn}.

\subsubsection{}
Now let $Z$ be the total space of 4 line bundles 
$$
Z = 
\begin{matrix}
\cL_1 \oplus \cL_2 \oplus \cL_3 \oplus \cL_4 \\
\downarrow\\
B
\end{matrix} \,, \quad \bigotimes \cL_i = \cK_B \,,
$$
over a smooth curve $B$, and let us look for 
$\bT$-invariant stable membranes in 
the class 
$$
[C] = 2 [B] \,. 
$$
Here $\bT\cong (\Ct)^3$ is the torus scaling the fibers with
determinant 1. 

\subsubsection{}
We will make the simplifying assumption that 
$$
\deg \cL_1 > 0 \ge \deg \cL_i\,,\quad i=2,3,4 \,,
$$
in which case $C$ can only double in the direction of $\cL_1$ as 
discussed in Section \ref{s_stabil} and all $\bT$-invariant
stable membranes have the form 
\begin{equation}
f_D : \,\, \bB_{\cL_1(-D)} 
\hookrightarrow S_{\cL(-D)}  
\xrightarrow{\,\,\, F_D \,\,\,} S_{\cL_1} \hookrightarrow Z\,, \label{f_D}
\end{equation}
where $D\in S^d B$ is an effective divisor of degree 
$$
0 \le d < \deg \cL_1 \,. 
$$
This range is restricted by the stability condition 
$
\chi\left(\bB_{\cL_1(-D)} \right)< 2  \chi(B) 
$. 

\subsubsection{}
The deformation theory of the map \eqref{f_D} may be described 
as follows 
$$
\Def(f_D) - \Obs(f_D) = \Def(F_D) + H^*\left(B,N_{\bB_{\cL_1(-D)}} S_{\cL_1(-D)}
+ f_D^* N_{S_{\cL_1}} Z\right)
$$
where
\begin{align*}
 N_{\bB_{\cL_1(-D)}} S_{\cL_1(-D)}  &= \cL_1(-D) +  \cL_1^2(-2D)\,,\\
f_D^* N_{S_{\cL_1}} Z &= (\cO + \cL_1^{-1}(D)) \otimes (\cL_2
+\cL_3+\cL_4) \,. 
\end{align*}

\subsubsection{}
The corresponding membrane integrals are particularly easy to 
compute for $B=\pP^1$ as then one can use the extra torus action 
on the base. They may be compared to the corresponding degree 2
PT integrals, which can also be computed by localization. As usual, 
there is, in fact, more than one PT check, as different 
tori may be designated as $\Cq$.

\subsection{Single interaction between smooth curves} 

\subsubsection{} 

Let $X_i$ be the components of $X=Z^{\Cq}$ and let 
$C_i \subset X_i$ be a collection of smooth reduced curves, possibly 
empty, in each component. 
As in Figure \ref{f_connectors}, let
 $f: C \to Z$ be a $\Cq$-invariant stable membrane such that 
$$
\pi_{\MM} (f(C)) \subset \sum \left[C_i\right]
$$
where the map $\pi_{\MM}$ 
of Section \ref{s_pi_DTM} keeps only those components that are fixed 
point-wise.  

We denote by $C'$ the other components of $C$ and 
focus here on the case when, unlike the situation depicted in 
Figure \ref{f_connectors},  $C'$ is closure of a single $\Cq$-orbit 
that flows from $p_2 \in X_2$ to $p_1\in X_1$. The general case, when $C'$ may be 
reducible or nonreduced, is expected to be covered by taking the 
symmetric algebra in \eqref{defbPhi}. 

\subsubsection{} 

There are 4 possible cases, namely 
$$
C = 
\begin{cases}
C' \,, \\
C' \cup C_1 \,, \\
C' \cup C_2 \,, \\
C' \cup C_1 \cup C_2 \,, 
\end{cases}
$$
corresponding to the 4 terms in the expansion of \eqref{defX21}. We
consider the last, most interesting case, assuming $C_1 \ne
\varnothing \ne C_2$.  We denote by 
$$
\Delta N =   N_CZ  - \big(N_{C'}Z + N_{C_1}Z + N_{C_2}Z \big) 
$$
the difference between the normal bundle to $Z$ and the normal 
bundles of its components. It may described as follows 
\begin{align}
\Delta N =  & \, 
 T_{p_1} C' \otimes T_{p_1} C_1  + T_{p_2} C' \otimes T_{p_2} C_2 
\notag \\
& - T_{p_1} Z + T_{p_1} C' + T_{p_1} C_1 \notag \\
& - T_{p_2} Z + T_{p_2} C' + T_{p_2} C_2 \label{NNC} \,,
\end{align}
where the first line corresponds to smoothing of 
the two nodes, while the second line is the condition of 
preserving the node at $p_1$ if it is not smoothed. 

\subsubsection{}

The contribution of $\Delta N$ to $\tO_\MM$ equals 
\begin{align*}
  \left(\det \Delta N \right)^{-1/2} \Sd \Delta N^\vee & = \Sd \left( -
    N^\vee_{C_1} X_1 \Big|_{p_1}- N_{C_2} X_2 \Big|_{p_2} \right) +
  O(q) \\
&= \eval^*\left(\cO_{C_2}^\vee \boxtimes \cO_{C_1} \right) + O(q) \,, 
\end{align*}
as $q\to 0$. We thus see that the form of the interaction described in
Section \ref{s_inter} is dictated already by the lowest $q$-term in 
the simplest interacting geometry.

\subsection{Higher rank DT counts}\label{e_rank} 

\subsubsection{}

By analogy with PT moduli spaces, one may consider $1$-dimensional 
sheaves $\cF$ with $r$ sections, that is, complexes of the form 
\begin{equation}
\cO_X^{r} \xrightarrow{\,\,s\,\,} \cF\,,
\label{PTr} 
\end{equation}
subject to the same stability conditions. They have a natural action
of $GL(r)$ by automorphisms of $\cO_X^r$. 

In contrast to the 
case $r=1$, the deformations of \eqref{PTr} for $r>1$ generally lead
to complexes not of the form \eqref{PTr}. This is a well-known 
phenomenon even if $X$ is a surface, where the points of the form
\eqref{PTr}   in the moduli space of all framed torsion-free sheaves
$\cG$ correspond to torsion-free sheaves $\cG=\Ker s$ with $\cG^{\vee \vee} \cong
\cO_X^r$, in other words, to instantons of zero size. 

\subsubsection{}

While constructing a proper moduli space,
with an $GL(r)$-action, that contains the deformations of \eqref{PTr}
is certainly an interesting problem with many potential applications, 
this problem remains currently open even for the simplest surface
$\C^2$. 

Instead, here we take a pragmatic approach and define 
higher-rank PT invariants by localization with respect to the 
maximal torus $\bA\subset GL(r)$. The corresponding fixed loci 
are direct sums 
\begin{equation}
\cO_X^{r} \xrightarrow{\,\,\oplus s_i \,\,} \cF = \bigoplus \cF_i \,,
\label{PTrs}
\end{equation}
and thus $r$-fold products of PT moduli spaces of $X$, with the
natural 
direct sum obstruction theory. To account for 
modification required in rank $r$, we define 
\begin{equation}
  \label{deftODTr}
  \tO_{\DT,r} = \textup{prefactor}  \, \, 
 \Ovir \otimes \left( \Kvir \otimes \bL^{\otimes r}_{\cL_1-\cL_2}
 \right)^{1/2}
\otimes \textup{cross-terms} \,. 
\end{equation}
The form of the prefactor changes to 
\begin{equation}
  \label{deftODT2r}
  \textup{prefactor} = 
(-1)^{(r \cL_2 + \cK,\beta)+r \chi} \, q^{(\beta,\cK)+\chi} \, 
Q^\beta \notag \,,
\end{equation}
where $\beta = \ch_2(\cF)$ and  $\chi=\chi(\cF)$. 

\subsubsection{}
The cross-terms in the deformation theory of \eqref{PTrs} decompose 
according to the weights of $\bA$, the term 
\begin{equation}
N_{ji} = \chi(\cF_i) + \chi(\cF_j^\vee) - \chi(\cF_j,\cF_i) \label{Nji}
\end{equation}
having the weight $a_i/a_j$. We have 
$$
\textup{cross-terms} = \bigotimes_{i\ne j} \tSd N_{ji} 
$$
where for a K-theory class $V$, we set, for brevity 
$$
\tSd V  =\left( \det V \right)^{-1/2} \otimes \Sd V^\vee  \,. 
$$
The argument of 
Section \ref{pull_square} is modified easily to show that the square 
root in \eqref{deftODTr}, including the square root present in the 
cross-terms, is well-defined modulo line bundles pulled back from 
the Chow variety.

\subsection{Engineering higher rank DT theory}\label{eng_Hr}

\subsubsection{} 

Let $Z_r$ be an $A_{r-1}$-surface fibration over $X$ as in Section
\ref{sAr}.  We label the components 
$$
Z^{\Cq} = X_1 \sqcup X_2 \sqcup \dots \sqcup X_r 
$$
of the fixed locus so that the unbroken curves flow from larger indices to smaller. 
With such labeling 
$$
N_{X_i} Z = \cL_1^{r} \cK^{i-r} \oplus \cL_2^{r} \cK^{1-i} \,, \quad
i=1,\dots, r\,, 
$$
where $\cK = \cK_X$. These have $\Cq$-weights $(q,q^{-1})$ by our
convention, although this $q$ is the $r$th power of the variable that 
originally acted on $Z_1$ before the quotient and the resolution. 

See Figure \ref{f_Z4} for a schematic representation of the geometry
of $Z_r$. 

\subsubsection{}
Our goal in this section is to prove the following

\begin{Proposition}\label{p_eng}
Assuming Conjecture \ref{sconj1}, the M2-brane index of $Z_r$ equals 
the rank $r$ Donaldson-Thomas partition function of $X$. 
\end{Proposition}

For this statement to make sense, one has to substitute
$GL(r)$-equivariant parameters for 
K\"ahler parameters of $Z_r$, in other words, one 
needs a surjective map 
$$
\gamma: \bA \twoheadrightarrow \TK(Z_r)  \,. 
$$
We start with the description of $\gamma$.

\subsubsection{}
To define $\gamma$, it suffices to give the images of the 
coordinate cocharacters 
$$
\delta_i : \Ct \to \bA 
$$
in the cocharacter lattice $H^2(X,\Z)$ of $\TK(Z_r)$. 

For $s\in \{\frac12,\frac32,\dots,r+\frac{1}{2}\}$, let
\begin{align*}
D_s &= \textup{attracting manifold of $X_{s-1/2}$} \\
      &= \textup{repelling manifold of $X_{s+1/2}$} 
\end{align*}
where attracting and repelling manifolds are defined for the action of
$q\to 0$. This is illustrated in Figure \ref{f_Z4}.

By construction, this means that 
\begin{equation}
D_s\Big|_{X_i} = 
\begin{cases}
c_1(\cL_2^r \cK^{1-i})\,, \quad &i = s-1/2 \,,\\
c_1(\cL_1^r \cK^{i-r})\,, \quad &i = s+1/2 \,, \\
0 \,,  \quad &\textup{otherwise} \,. 
\end{cases}\label{DsX}
\end{equation}
We set 
$$
\gamma(\delta_i) = \frac12 \sum_{s<i} D_s - \frac12 \sum_{s>i} D_s \,.
$$

\psset{unit=1.8 cm}
\psset{yunit=1.2 cm}
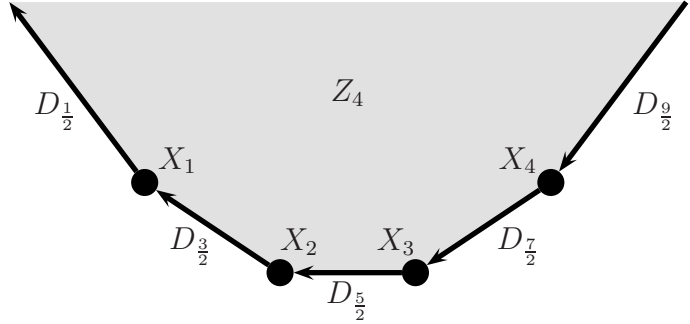
\begin{figure}[!htbp]
  \centering
   \begin{pspicture}(0,-0.5)(5,3)
\newgray{lg}{.88}
\pspolygon*[linecolor=lg](0,3)(1,1)(2,0)(3,0)(4,1)(5,3)
\pnode(0,3){A}
\cnode*(1,1){.1}{X1}
\cnode*(2,0){.1}{X2}
\cnode*(3,0){.1}{X3}
\cnode*(4,1){.1}{X4}
\pnode(5,3){B}
\ncline[linewidth=2pt]{<-}{A}{X1}
\ncline[linewidth=2pt]{<-}{X1}{X2}
\ncline[linewidth=2pt]{<-}{X2}{X3}
\ncline[linewidth=2pt]{<-}{X3}{X4}
\ncline[linewidth=2pt]{<-}{X4}{B}
\rput[lb](1.1,1.1){$\displaystyle X_1$}
\rput[lb](2,0.2){$\displaystyle X_2$}
\rput[rb](3,0.2){$\displaystyle X_3$}
\rput[rb](3.9,1.1){$\displaystyle X_4$}
\rput[tr](0.5,2){$\displaystyle D_{\frac12}$}
\rput[tr](1.5,0.5){$\displaystyle D_{\frac32}$}
\rput[t](2.5,-0.1){$\displaystyle D_{\frac52}$}
\rput[tl](3.6,0.5){$\displaystyle D_{\frac72}$}
\rput[tl](4.6,2){$\displaystyle D_{\frac92}$}
\rput[c](2.5,2){$\displaystyle Z_4$}
\end{pspicture}
 \caption{The moment map for the fiberwise $(\C^\times)^2$-action
   offers a schematic representation of $Z_4$. 
The 3-folds $X_i$ and the divisors $D_s$ are mapped to 
vertices and edges, respectively. The arrows indicate $q\to 0$
limits.}
  \label{f_Z4}
\end{figure}

\subsubsection{}
It is easy to describe the dual map 
\begin{equation}
\gamma^*: H_2(Z_r,\Z) \hookrightarrow \bA^\wedge\,,\label{convA}
\end{equation}
between the character groups. The lattice 
$\bA^\wedge$ is spanned by coordinate characters
$\varepsilon_i$,  where $a^{\varepsilon_i} =a_i$. They form the 
basis dual to $\{\delta_i\}$. 

The classes of unbroken curves $C_{ij}$ from 
$X_j$ to $X_i$ 
are mapped to positive roots 
$$
\gamma^*\left([C_{ij}]\right)  =  \varepsilon_i -  \varepsilon_j \,,
\quad i < j \,. 
$$
while on homology classes supported on $\bigsqcup X_i$ the map 
$\gamma^*$ is given by the following formula.  

Fix a curve class $\beta_k$ in each of the $X_k$'s and let 
\begin{equation}
\vec \beta = \sum_k \iota_{k,*} \beta_k \in H_2(Z,\Z) \label{vecb} 
\end{equation}
be their union in $Z$, where $\iota_k$ is the inclusion of $X_k$. 
{}From \eqref{DsX}, we have 
\begin{equation}
\gamma^*(\vec\beta ) = \frac{r}{2} \sum_i (L_1-L_2,\beta_i) \, 
\varepsilon_i -  \frac{1}{2} \sum_{i<j} (\cK,\beta_i+\beta_j) ( \varepsilon_i -
\varepsilon_j) \,. \label{gvbe} 
\end{equation}

\subsubsection{}
Now let the PT data on $\bigsqcup X_i$ be specified by collections 
of sheaves $\cF_i$ with sections $s_i$ as in \eqref{PTrs}.  We define 
$\beta_i = \ch_2 \cF_i \in H_2(X_i,\Z)$, denote by \eqref{vecb} the
union in $Z$ of these classes, and set
$$
\beta = \sum \beta_i \in H_2(X,\Z)\,.
$$
By the main conjecture, the contribution of \eqref{PTrs} to the
membrane index equals the product of a certain prefactor, virtual class 
contribution, and the interaction $\Phi$. The prefactor equals, including 
the replacement of K\"ahler parameters by equivariant ones, 
\begin{align}
 \textup{prefactor} = & (-1)^{\chi(\cF)+\sum(r \cL_1 + (i-r)
    \cK,\beta_i)} \notag \times \\
& q^{\chi(\cF) + (\beta,\cK)/2} \, Q^\beta \, a^{\gamma^*(\vec\beta)}
\,.  
 \label{prefC}
\end{align}
The virtual class contribution equals 
\begin{equation}
  \label{vcC}
  \begin{matrix}
 \textup{virtual class} \\
\textup{contribution} 
  \end{matrix}  = 
\boxtimes\,\, \Ovir \otimes \left( \Kvir \otimes \bL_{\cL_1^r \cK^{i-r}-
\cL_2^r \cK^{1-r}} \right)^{1/2} \,.
\end{equation}
Finally, in the interaction terms, we discard the perturbative terms as 
discussed in Section \ref{s_pert}, and for the remainder we get 
the following identification 
$$
\cX'_{ji} = N_{ji} + \overline{N_{ij}}\,,
$$
where $N_{ji}$ was defined in \eqref{Nji} and bar denotes dual. 
Therefore, we have 
\begin{equation}
  \label{crossC}
  \textup{interaction $\Phi$} = \bigoplus_{i<j} \Sd 
\left(\frac{a_i}{a_j} (N_{ji} +
  \overline{N_{ij}})\right)
\,. 
\end{equation}
This is clearly beginning to look like higher rank DT theory, and we
will now systematically check the agreement. 

\subsubsection{}

We start with the following identity 
\begin{equation}
\frac{\Sd 
\left(\frac{a_i}{a_j} (N_{ji} +
  \overline{N_{ij}})\right)} 
{\tSd 
\left(\frac{a_i}{a_j} N_{ji} +
  \frac{a_j}{a_i} N_{ij} \right)} = 
(-1)^{\rk N_{ji}} 
\left(\frac{a_i}{a_j}\right)^{\!\!
- \frac12 (\rk N_{ij} + \rk
  N_{ji})} 
\left(\frac{{\det N_{ij}}}{{\det N_{ji}}}\right)^{1/2}
\,. \label{SoverS} 
\end{equation}
We have 
$$
\rk N_{ji} = \chi(\cF_i) - \chi(\cF_j) - (\cK,\beta_j) \,,
$$
and so, in particular, 
$$
- \frac12 (\rk N_{ij} + \rk
  N_{ji}) = \frac12 (\cK,\beta_j+\beta_i) \,.
$$
Thus from \eqref{gvbe} we conclude 
\begin{equation}
  \label{powa}
 a^{\gamma^*(\vec \beta)} \, \left(\frac{a_i}{a_j}\right)^{\!\!
- \frac12 (\rk N_{ij} + \rk
  N_{ji})} = \prod a_i^{\frac12 r(\cL_1-\cL_2,\beta_i)}  \,. 
\end{equation}
We observe that this is the natural $\bA$-weight of the bundle 
$\bL^{r/2}_{\cL_1-\cL_2}$ that appears in \eqref{deftODTr}. 

Recall
that we define higher rank DT invariants as integrals over
$\bA$-fixed loci, and thus all line bundle contributions have a 
line bundle part, which is defined in DT theory of $\bigsqcup X_i$
and an $\bA$-character part that comes from converting K\"ahler 
parameters to equivariant ones.

A similar check finds the agreement between the minus signs and 
the sign in the prefactor in \eqref{deftODT2r}. 

\subsubsection{}

We now turn to the determinants in the right-hand side of
\eqref{SoverS}. By Proposition \ref{p_symm} and Serre duality, we have 
\begin{equation}
\frac{{\det N_{ij}}}{{\det N_{ji}}} = \frac{\det \chi(\cF_j\otimes
  (\cO -\cK))}{\det \chi(\cF_i\otimes
  (\cO -\cK))} \,.\label{eq:6}
\end{equation}
Thus 
\begin{equation*}
\boxtimes \, \bL_{\cL_1^r \cK^{i-r}-
\cL_2^r \cK^{1-r}} 
\otimes \bigotimes_{i<j} \frac{{\det
    N_{ij}}}{{\det N_{ji}}} = 
\boxtimes \, \bL_{\cG_i}
\end{equation*}
where
$$
\cG_i = \cL_1^r \cK^{i-r}-
\cL_2^r \cK^{1-r} - (2i-r-1)(\cO-\cK) \,, \quad i=1,\dots,r \,. 
$$
A direct computation shows that 
$$
\cG_i = r (\cL_1 - \cL_2) + \dots 
$$
where dots stand for a $K$-theory class of codimension 3 in $X$ which 
thus does not affect line bundles of the form $\bL_\cG$. 

This completes the proof of Proposition \ref{p_eng}.

\section{Existence of square roots} 
\label{s_sqrt}

\subsection{Symmetric bundles on squares}

We start with the following general observation. 

\begin{Lemma}
Let $Y$ be a algebraic variety and let $L$ be a line
bundle on $Y \times Y$ such that 
$$
(12)^* L \cong L
$$
where $(12)$ is the permutation of factors. Then the restriction 
$L_\Delta$ 
of $L$ to the diagonal $\Delta \subset Y \times Y$ has a square root. 
\end{Lemma}

\noindent
Our original claim was weaker. We are very grateful to 
Davesh Maulik who pointed out to us that the argument works in full
generality presented here. 

\begin{proof}
{}From an \'etale exact sequence of sheaves
$$
1 \to \{\pm 1\} \to \cO_Y^* \xrightarrow{f\mapsto f^2} \cO_Y^* \to 1 
$$
we have an exact sequence of groups 
$$
H^1(Y,\Z/2) \to \Pic Y \xrightarrow{\cL\mapsto \cL^{\otimes 2}} \Pic Y
\to H^2(Y,\Z/2)\,,
$$
where the last map is the reduction of $\cL \mapsto c_1(\cL)$ modulo
$2$. Therefore, a line 
bundle on $Y$ has a square root if and only if its first Chern class
is divisible by $2$ in $H^2(Y,\Z)$. 

Since 
$$
\textup{torsion} \left(H^k(Y,\Z)\right) \cong 
\textup{torsion} \left(H_{k-1}(Y,\Z)\right)
$$
the groups $H^0(Y,\Z)$ and $H^1(Y,\Z)$ are torsion-free. Therefore, 
K\"unneth decomposition takes the form 
\begin{equation}
H^2(Y \times Y,\Z) = \bigoplus_{i+j=2} H^i(Y,\Z) \otimes H^j(Y,\Z) \,. 
\label{Kunn}
\end{equation}
Assuming $Y$ is connected, the symmetry of $L$ implies that 
$$
c_1(L) = \alpha \otimes 1 +  \beta + 1 \otimes \alpha\,, \quad 
\alpha \in H^2(Y,\Z),\, \beta\in \Lambda^2 H^1(Y,\Z)\,,
$$
in the decomposition \eqref{Kunn}. The restriction to the diagonal 
of the middle piece is the map 
$$
\beta_1 \otimes \beta_2 \mapsto \beta_1 \cup \beta_2 
$$
and from the skew-symmetry of cup product on $H^1(Y,\Z)$ we conclude that
$c_1(L_\Delta)$ 
is even. 
\end{proof}

\subsection{Square roots in DT theory} 

\subsubsection{} 

Symmetric line bundles on products appear
naturally in the DT theory of 3-folds.
Let $\cF$ be the universal family of the $1$-dimensional sheaves
over $\PT(X)$. Consider the line bundle 
$$
\bL_{12} = \det \chi(\cF_1,\cF_2)
$$
over the product of two PT moduli spaces. 

\begin{Proposition}\label{p_symm} 
There is a canonical isomorphism $\bL_{12} \cong \bL_{21}$. 
\end{Proposition}

We expect the same symmetry to hold for 
Donaldson-Thomas moduli space of X in any stability chamber. 
The proof below will have to be modified to account 
for $0$-dimensional subsheaves in $\cF$.

Observe that this statement is consistent with the following 
special case of the Serre duality. Suppose $\cK_X$ is trivial 
and let $\bk$ be the weight of $\Aut X$ action on $\cK_X$. 
Then Serre duality gives
\begin{multline}
   \det \chi(\cF_1,\cF_2) =  \det \chi(\cF_2,\cF_1 \otimes \cK_X) =\\
=  \det \chi(\cF_2,\cF_1) \otimes \bk^{\rk \chi(\cF_2,\cF_1)} = 
\det \chi(\cF_2,\cF_1) \,.
\label{deghomchi}
\end{multline}
It is clear that the Proposition follows, by Serre duality, from the 
following 

\begin{Lemma}\label{l_otimesL} 
For any line bundle $\cL$ on $X$ we have
$$
\det \chi(\cF_1,\cF_2) = \det \chi(\cF_1,\cF_2\otimes \cL)\,, 
$$
canonically. 
\end{Lemma}

\begin{proof}
By writing $\cL$ as a ratio of two very ample line bundles, we may 
reduce to the case when $\cL$ is very ample. Let $s$ be a  generic 
section of $\cL$. The choice 
of $s$ in not unique and the dependence on the choice of $s$ 
will be analyzed later. Consider the $0$-dimensional sheaf 
$$
\cG = \Coker \left(\cF_2 \xrightarrow{s} \cF_2 \otimes \cL\right) \,.
$$
It has a 
canonical filtration by direct sums of sky-scraper sheaves 
$\cO_x$, $x\in X$.  For any sheaf $\cF$ on $X$ we have 
$$
\det \chi(\cF,\cO_x) = \left(\det \cF\right)^*_x \,,
$$
by taking a locally free resolution of $\cF$. Since $\cF_1$ is 
$1$-dimensional, 
$$
\det \cF_1 = \cO_X
$$
which gives an isomorphism 
$$
\phi_s: \det \chi(\cF_1,\cF_2)  \to \det \chi(\cF_1,\cF_2\otimes
\cL)\,. 
$$
It remains to analyze the dependence of this isomorphism on $s$. 

Denote
$$
\Delta \subset H^0(X,\cL)
$$
the set of sections $s$ for which $\cG$ fails to be $0$-dimensional. 
This is a conical subset of codimension $>1$. For any $s_0 \in
H^0(X,\cL)\setminus \Delta$ the function $\phi_s \phi_{s_0}^{-1}$ 
is homogeneous in $s$ of degree
$$
\deg_s \phi_s \phi_{s_0}^{-1}  = \rk \chi(\cF_1,\cF_2) = 0 
$$
and regular away from $\Delta$, hence identically $1$. 
\end{proof}

\subsubsection{}\label{pull_square} 

We have the following 

\begin{Proposition}\label{p_Ksquare} 
For any $\cL_1$ and $\cL_2$ such that $\cL_1\otimes \cL_2 = 
\cK_X$ we have 
$$
\Kvir \otimes \bL_{\cL_1-\cL_2} = \textup{square} \otimes 
\bL_{c_1(\cL_1) \cap c_1(\cL_2)} \,,
$$
where the last term in pulled back by the Hilbert-Chow map. 
\end{Proposition}

\begin{proof}
We first note that for $\cL_1=\cO_X$ and $\cL_2=\cK_X$ we have, 
by Serre duality 
$$
   \Kvir \otimes \bL_{\cO-\cK} = \det 
\big[ \chi(\cF,\cF) - 2 \chi(\cF\otimes \cK_X)
\big] \,,
$$
which is a square by Proposition \ref{p_symm}. On the 
other hand, 
$$
\bL_{\cL_1-\cL_2} \otimes \bL_{\cO-\cK}^{-1} = 
\bL_{\cL_1-\cO}^2  \otimes \bL_{(\cO-\cL_1)(\cO-\cL_2)} \,.
$$
Write 
$$
\cL_i = \cA_i \, \cB_i^{-1}\,, \quad i=1,2\,,
$$
where $\cA_i$ and $\cB_i$ are very ample. Then 
\begin{align*}
(\cO-\cL_1)(\cO-\cL_2) = \cA_1 \cA_2 \Big[
&(1-\cA_1^{-1})(1-\cA_2^{-1}) - (1-\cA_1^{-1})(1-\cB_2^{-1})\\
-&(1-\cB_1^{-1})(1-\cA_2^{-1}) + (1-\cB_1^{-1})(1-\cB_2^{-1}) \Big]
\,. 
\end{align*}
By Lemma \ref{l_otimesL}, all terms in the right-hand side produce 
bundles $\bL_{\cO_C}$ where $C\subset X$ is a complete intersection of 
two very ample divisors. By construction, such bundles are pulled 
back by the Hilbert-Chow map. Clearly, the resulting rational 
equivalence 
class of curves equals $c_1(\cL_1) \cap c_1(\cL_2)$. 
\end{proof}

\subsection{Square roots in M-theory} 

\subsubsection{}

Since the moduli space of stable membranes is still under
construction, we restrict ourselves here to numerical checks
under simplifying assumptions. 

\begin{Proposition}
If $\phi: B \to \MM(Z)$ is a map of a smooth curve
to the locus of local complete intersections then 
$$
\deg \phi^*\Kvir \equiv \int_{\textup{sweep of $B$}} c_2(Z)  
\mod 2 \,. 
$$
\end{Proposition}

\begin{proof}
Let 
$$
S \subset B \times Z \times \pP^N
$$
be the surface corresponding to the map $\phi$, where 
$\pP^N$ is the auxiliary projective space as in Section
\ref{s_mapsf}. By hypothesis, $S$ is locally a complete 
intersection, hence has a normal bundle $N_S$.   We set 
$$
N = N_{S} - T \pP^N, 
$$
this is a rank 4 bundle on $S$. From definitions, 
$$
\phi^*\Kvir = \left( \det \pi_* N \right)^{-1} 
$$
where $\pi: S \to B$ is the projection. By Grothendieck-Riemann-Roch, we have 
$$
\ch \pi_* N  = \pi_*  \left(\ch  N \,  \frac{\Td S}{\Td B}
\right) = \pi_*  \left(\ch  N \,  \frac{\Td Z}{\Td N} 
\right) \,. 
$$
Since $\rk N = 4$, we have
$$
\frac{ \ch N} {\Td N}  = 4 - c_1(N) + \frac{4}{3} \, 
\ch_2(N) 
$$
while 
$$
\Td Z = 1+ \frac{1}{12} c_2(Z) 
$$
since $c_1(Z)=0$. Putting everything together, we see that 
\begin{align*}
\deg \pi_* N &= \frac{1}{3} \int_S 
\left( c_2(Z) + 4 \ch_2(N) \right) \\
&\equiv  \int_S c_2(Z)  \mod 2 
\end{align*}
because $2\ch_2$ is an integral characteristic class and division by 3
does not affect parity. 
\end{proof}

\subsubsection{}

We now compare the parity computations in DT and M-theories. 
We consider the case when 
$$
Z = 
\begin{matrix}
\cL_1 \oplus \cL_2  \\
\downarrow\\
X
\end{matrix} \,, \quad \cL_1 \otimes \cL_2 = \cK_X \,,
$$
is a rank 2 bundle over a 3-fold $X$. We have the following 

\begin{Proposition}
 For $Z$ as above, 
$$
c_2(Z) \equiv c_1(\cL_1) \, c_1(\cL_2) \mod 2 \,. 
$$
\end{Proposition}

\begin{proof}
We have 
$$
c_2(Z) - c_1(\cL_1) \, c_1(\cL_2) = 
c_2(X) - c_1^2(X) \equiv c_2(X) +c_1^2(X) = 12 \Td_2(X) \,,
$$
so it remains to see that 
$$
6 \int_D \Td_2(X) \in \Z
$$
for any complete divisor $D\in X$. This is an easy consequence of the 
integrality of the function 
$$
k\mapsto \dim \chi(\cO_D(kD))
$$ 
and the Hirzebruch-Riemann-Roch formula.
\end{proof}

\section{Refined invariants}\label{s_refined} 

\subsection{Actions scaling the 3-form}

\subsubsection{}\label{ss_sc1} 

By far the most popular manifolds $Z$ for M-theory constructions 
have the form 
$$
Z = X \times \C^2 
$$
where $X$ is a Calabi-Yau 3-fold and $\C^2\cong \R^4$ is the 
space-time of our everyday experience (which may be 
replaced by an $A_n$ surface in everything that follows).  

 As a $\Cq$ torus 
for such $Z$ we can take the maximal torus 
$$
\begin{bmatrix}
 q \\ & q^{-1} 
\end{bmatrix} \subset SL(2) 
$$
acting
on the $\C^2$ factor.
Since 
$\cL_1 = \cL_2 = \cO_X$, we have 
$$
\tO_{\PT} = (-q)^\chi \, Q^\beta \, \Ovir \otimes \Kvir^{1/2} \,. 
$$
In this entire section, we will work with individual components of the 
DT moduli spaces and will drop the prefactor $(-q)^\chi \, Q^\beta$ for 
brevity. 

\subsubsection{}

Let $G$ be a connected group acting on $X$ and let 
$\bk$ be the determinant of this action, that is, 
$$
\bk = \textup{weight}  \left(\Lambda^{3} T^{1,0} X \right)\,. 
$$
The letter $\bk$ is supposed to remind of the canonical class $\cK_X$, 
except it is the \emph{inverse} of the $G$-weight of $\cK_X$. 

Let $\Gk$ 
be the minimal cover of $G$ 
on which the character $\bk^{1/2}$ is defined. 
We define 
$$
\Gk \hookrightarrow G_q = \Aut(Z,\Omega^5)^{\Cq} 
$$
by 
$$
 g \mapsto \left(g, \begin{bmatrix}
 \bk(g)^{-1/2} \\ & \bk(g)^{-1/2} 
\end{bmatrix} \right) \,.
$$
The square root $\bk^{1/2}$ is needed to make $\cL_1 = \cL_2$ 
as $\Gk$-equivariant line bundles. 

The results of this 
section will be particularly interesting if $\bk$ is
\emph{nontrivial}. Since $\Gk$ acts trivially on cohomology, we have 
$$
\bk \ne 1  \Rightarrow [\Omega^3_X] = 0 \in H^3(X) \,, 
$$
and so $X$ has to be noncompact for this to happen. Examples of 
$X$ with $\bk\ne 1$ include toric Calabi-Yau varieties, local curves, 
and local surfaces. 

We note that even for noncompact $X$ the PT moduli spaces may 
very well be compact which will be important below. 

\subsubsection{}

The main result of this section is the following 

\begin{Theorem}\label{t_kappa} 
For any Calabi-Yau 3-fold $X$, the sheaf $\tO_{\PT}$ has a
canonical $\Gk$-equivariant structure. If $\cM$ is a proper 
component of $\PT(X)$, then 
$$
\chi(\cM,\tO_{\PT} ) \in \Z\left[\bk^{\pm 1/2}\right] \subset 
K_{\Gk} (\pt) \,,
$$
that is, the $\Gk$ action on $\chi(\cM,\tO_{\PT} )$ factors 
through the character $\bk^{1/2}$. Further, the polynomial
$\chi(\cM,\tO_{\PT} )$ is symmetric with respect to 
$$
\bk^{1/2} \mapsto \bk^{-1/2}\,.
$$
\end{Theorem}

For $\bk\ne 1$, we conjectured the polynomials from Theorem 
\ref{t_kappa} to agree with the motivically refined DT invariants 
studied in \cite{BBS,BJM,KS} and many other papers. This conjecture
has been proved by Davesh Maulik in \cite{Dmot}. 

\subsubsection{}

The conclusions of Theorem \ref{t_kappa} hold, in fact, for the sheaf 
$$
\tO_{\vir} = \Ovir \otimes \Kvir^{1/2}
$$
of any 
symmetric perfect obstruction theory on which a group $\Gk$ 
acts scaling the symmetry of the obstruction theory
\begin{equation}
\Obs \cong \Def^\vee \otimes \C(\bk) \label{symm_bk}\,, 
\end{equation}
where $\C(\bk)$ is a $1$-dimensional representation of weight 
$\bk$. 

For example, in Section \ref{s_Hodge}  one substitutes $t=\bk$ to 
see that 
\begin{equation}
\chi \left(\Ovir \otimes \Kvir^{1/2}\right) = 
\sum_{p,q} (-1)^{p-d/2} (-\bk)^{q-d/2} H^p(\Omega^q M)  
\label{hpk}
\end{equation}
where $H^p(\Omega^q M)$ are trivial representations of $\Gk$
and $d = \dim M$. 
The symmetry 
$$
h_{p,q} = h_{d-p,d-q} 
$$
of the Hodge diamond of $M$ implies the 
$\bk^{1/2} \mapsto \bk^{-1/2}$ symmetry of the polynomial
\eqref{hpk}. 

\subsubsection{} 

The basic properties of the index \eqref{hpk} are very classical 
and can be seen from many different angles, in particular, from the 
point of view of Morse theory as in \cite{WittenMorse}. We will find 
the Morse theory point of view to be 
very useful for general symmetric perfect 
obstruction theories. 

The triviality of the $\Aut(M)_0$ action on 
$H^p(\Omega^q M)$ is an example of \emph{rigidity}, see in particular 
\cite{BT}.  The argument used in Section \ref{s_rig} would be very 
familiar to anyone who read \cite{BT}.

\subsection{Localization for $\bk$-trivial tori}
\label{s_local}

\subsubsection{Equivariant structure on square roots}

In this section we assume $\cM$ to be a projective scheme with 
a symmetric perfect obstruction theory and an action of an  
algebraic group 
$G$ that scales the symmetry of the obstruction theory by 
a character $\bk$ as in \eqref{symm_bk}.  As before, we denote by 
$\Gk$ the minimal cover of $G$ on which the character 
$\bk^{1/2}$ is defined. 

We assume that 
the line bundle $\Kvir$ is a square in the nonequivariant 
Picard group of $\cM$. This holds for PT moduli spaces by 
Proposition \ref{p_Ksquare}.  Let $\Kvir^{1/2}$ be a choice 
of a square root of $\Kvir$. 

\begin{Proposition}\label{p_equiv12}
There is a canonical $\Gk$-action on $\Kvir^{1/2}$. 
\end{Proposition}

\begin{proof}
Any line bundle $\cL$ on a projective scheme $\cM$ is 
uniquely determined by the module
$$
\Gamma(\cL) = \bigoplus_{n\gg 0} H^0(\cL(n))
$$
over the homogeneous coordinate ring of $\cM$ and an 
equivariant structure on $\cL$ is an equivariant structure
on this module. 

The fiberwise squaring map $\Kvir^{1/2} \to \Kvir$ gives 
$\Kvir^{1/2}$ and the corresponding module a canonical 
$\fg$-module structure, where 
$$
\fg = \Lie G = \Lie \Gk \,. 
$$
We claim that 
\begin{equation}
\textup{weights} \, \Gamma(\Kvir^{1/2}(n)) 
\subset \textup{weights}(G) + 
\Z \bk^{1/2} = \textup{weights}(\Gk) \,. \label{eq:7}
\end{equation}
Indeed, let $\bT\subset G$ be the maximal torus and $\ft = \Lie
\bT$.  By the duality between deformations and
obstructions, the $\ft$-weight of $\Kvir^{1/2}$ at any fixed point 
is a weight of $\Gk$. Since sections of $\Kvir^{1/2}(n)$ are uniquely 
determined by their series expansion at fixed points, \eqref{eq:7}
follows. 

Thus, $H^0(\Kvir^{1/2}(n))$ is a finite-dimensional $\fg$-module 
satisfying \eqref{eq:7}, hence integrates to a $\Gk$-module. 
\end{proof}

In particular, the proposition makes 
$$
\tO_\cM = \Ovir \otimes \Kvir^{1/2}
$$
a $\Gk$-equivariant sheaf on $\cM$. 

\subsubsection{}

Fix a maximal torus $\bT \subset G$ and a subtorus 
$\bA \subset \bT$ in the kernel of $\bk$. This, in particular, 
means that $\bA$ acts canonically on $\Kvir^{1/2}$. 
Let $\cM^\bA\subset \cM$ be the locus of $\bA$-fixed points.
Equivariant localization for virtual Euler characteristics takes the 
following form \cite{GP,FG}. 

Restricted
to $\bA$, the obstruction theory decomposes
$$
\left(\Def-\Obs\right)\Big|_{\cM^\bA} =  
\left(\Def-\Obs\right)^\textup{fixed} \oplus 
\left(\Def-\Obs\right)^\textup{moving}
$$
where the moving part is the one that transforms in nontrivial
representations
of $\bA$. The triviality of $\bk$ on $\bA$ implies 
$$
\Obs ^\textup{fixed} \cong \left(
\Def ^\textup{fixed}\right)^\vee \otimes \C(\bk)\,,
$$
and similarly for the moving part. In particular, the fixed part of 
the obstruction theory defines a symmetric perfect obstruction 
theory for the fixed locus, with its own virtual structure sheaf. 
We have 
\begin{equation}
\chi(\cM,\tO_\cM) = \chi\left(\cM^\bA, \cO_{\cM^\bA,\textup{vir}} 
\otimes \Kvir^{1/2}\big|_{\cM^\bA} \otimes \Sd \cN^\vee \right)\,,
\label{vir_loc}
\end{equation}
where 
$$
\cN = \left(\Def-\Obs\right) ^\textup{moving} 
$$
is the virtual normal bundle.  

Our next goal is to rewrite the formula \eqref{vir_loc} using a
specific choice of a square root of the virtual canonical 
bundle of the fixed loci.

\subsubsection{}

The $\bA$-weights that appear in $\cN$ 
partition $\Lie \bA$ into finitely many chambers. We fix 
one chamber $\fC$, this separates all weights into positive and 
negative, so we can write 
$$
\cN =  
\cN_{+} \oplus 
\cN_{-} 
$$
with 
$$
\cN_{-} = - \cN_{+}^\vee 
\otimes \C(\bk) \,. 
$$
We have the following 

\begin{Lemma}
The nonequivariant line bundle
 $\det \cN_+$ does not depend on the choice of 
$\fC$ and satisfies 
$$
\det \cN = (\det \cN_+)^2 \,. 
$$
\end{Lemma}

\begin{proof}
Let $\fC_1$ and $\fC_2$ be two chambers and denote by 
$\cN_{\pm,\pm}$ the parts of $\cN$ spanned by characters
of given sign on $\fC_1$ and $\fC_2$. Then 
$$
\cN_{+,-} = - \cN_{-,+}^\vee
$$
whence 
$$
\det \cN_{+,-} = \det \cN_{-,+}
$$
which implies the independence. In particular $\det \cN_+ =
\det \cN_-$ and hence 
$$
\det \cN = \det \cN_+  \det \cN_- = (\det \cN_+)^2 \,. 
$$
\end{proof}

We give $\det \cN_+$ the canonical 
$\bT_{\bk}$-equivariant structure provided 
by the proof of Proposition \ref{p_equiv12} and set 
\begin{equation}
\cK_{\cM^\bA}^{1/2} = \det \cN_+
\otimes \Kvir^{1/2} \big|_{\cM^\bA} \label{eq:9}
\end{equation}
This provides a consistent choice of the square-root for the fixed
loci which depends on the choice of the square root on $\cM$.  

\subsubsection{}

We denote 
$$
\tO_{\cM^\bA} = \cO_{\cM^\bA,\textup{vir}} \otimes
\cK_{\cM^\bA}^{1/2} 
$$
and let $\rho$ be a map from equivariant K-theory to its completion
such that 
$$
\rho(A\oplus B) = \rho(A) \otimes \rho(B) 
$$
and 
$$
\rho(\cL) = \frac{\bk^{1/2}-\bk^{-1/2} \cL}{\cL-1} 
$$
for a line bundle $\cL$. Rewriting \eqref{vir_loc} using \eqref{eq:9} 
and duality, we obtain the following 

\begin{Proposition}\label{p_rho}
For any chamber $\fC\subset \Lie \bA$, we have 
\begin{equation}
\chi(\cM,\tO_\cM) = \chi\left(\cM^\bA, \tO_{\cM^\bA} 
\otimes \rho(\cN_{+}) \right) \,. 
\label{vir_loc2}
\end{equation}
\end{Proposition}

\subsection{Morse theory and rigidity}

\subsubsection{Virtual index of a fixed component}

We denote 
\begin{align}
  \sind_\fC &= 
\rk \cN_+\notag \\
&= \rk \Def_+ - \rk \Def_- \label{sind} \,,
  \end{align}
where the equality between the two lines follows from the 
duality between $\Obs_+$ and $\Def_-$; this also shows the 
independence of the second line on a particular representative
of the obstruction theory. 

Clearly, $\sind_\fC$ is a locally constant function on $\cM^\bA$ 
which we call the virtual index of a fixed component. 

\subsubsection{Rigidity}\label{s_rig}

\begin{Proposition}\label{p_rig}
The $\Gk$ action on $\chi(\cM,\tO_\cM)$ factors through the 
character $\bk^{1/2}$. In fact, 
\begin{equation}
  \label{eq:11}
  \chi(\cM,\tO_\cM) = \chi\left(\cM^\bA, \tO_{\cM^\bA}  \otimes
    (-\bk^{1/2})^{\sind_\fC} \right)
\end{equation}
for any chamber $\fC$ in the Lie algebra of
$$
\bA = \bT \cap \Ker \bk \,,
$$
where $\bT$ is a maximal torus of $G$. 
\end{Proposition}
\begin{proof}
It is enough to show that the $\bT_{\bk}$ action on $\chi(\cM,\tO_\cM)$
factors through the character $\bk^{1/2}$.  Since $\cM$ is compact, 
the $\bT_{\bk}$-character of $\chi(\cM,\tO_\cM)$ is a Laurent 
polynomial which, we claim, is constant on the $\bA$-cosets.  

Let $\sigma: \C^\times \to \bA$ be a generic homomorphism and let 
$\fC$ be the chamber containing $d\sigma \in \Lie\bA$. 
Then all characters that appear in $\cN_+$ go to zero on 
$\sigma(z)$ and $z\to 0$ and to infinity as $z\to\infty$. Therefore, for any $t\in\bT$ we have 
\begin{alignat*}{2}
   \chi(\cM,\tO_\cM)\big|_{t \sigma(z)} 
 &= \chi\left(\cM^\bA, \tO_{\cM^\bA}  \otimes
    (-\bk^{1/2})^{\sind_\fC} \right)  + O(z)\,, \quad &&z\to 0 \\
&= \chi\left(\cM^\bA, \tO_{\cM^\bA}  \otimes
    (-\bk^{1/2})^{-\sind_{\fC}} \right)  + O(z^{-1})\,, \quad &&z\to \infty
  \,. 
\end{alignat*}
Since this a Laurent polynomial in $z$, it is a constant equal to 
its value at either $0$ or $\infty$. Since $\sigma$ was generic, 
the claim follows. 
\end{proof}

\subsubsection{Conclusion of the proof of Theorem \ref{t_kappa}}

It remains to show the $\bk\mapsto\bk^{-1}$ symmetry. 
Since the virtual dimension of any self-dual theory is $0$,
from 
the weak Serre duality theorem of \cite{FG} we get
$$
\chi(\tO_\cM) = \chi(\tO_\cM)^\vee
$$
and since $\bk^\vee = \bk^{-1}$ we are done. 

\section{Index vertex and refined vertex} \label{s_index_vertex}

\subsection{Toric Calabi-Yau 3-folds}

\subsubsection{}

In this section, we specialize the discussion of Section
\ref{s_refined} to toric Calabi-Yau 3-folds $X$. For such $X$, the 
torus 
$$
\bA = \bT \cap \Ker \bk \cong \left(\C^\times\right)^2 
$$
acts with isolated fixed points on the Hilbert scheme of 
curves, and hence K-theoretic DT invariants of $X$ may be 
given by a combinatorial formula of the same flavor as 
the localization formula for cohomological DT invariants of $X$ \cite{mnop}, 
known to many in the formalism of the topological vertex \cite{AKMV,ORV}. 

\subsubsection{}

Let 
$$
\Delta(X) \subset \left(\Lie \bT\right)^*
$$
be toric polyhedron, that is, 
the image of the moment map for some K\"ahler class on $X$.
The projection of its 1-skeleton to $\left(\Lie \bA\right)^*$ is 
known as the toric diagram of $X$. The combinatorial type of 
these objects does not depend on the choice of the K\"ahler class. 

\begin{figure}[!htbp]
  \centering
   \scalebox{0.6}{\includegraphics{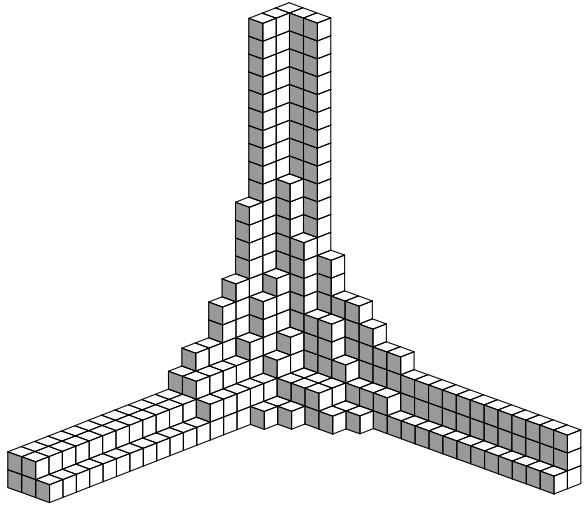}}
 \caption{A 3-legged 3-dimensional partition}
  \label{f_part}
\end{figure}

By a 3-dimensional partition with legs, we mean an object of 
kind shown in Figure \ref{f_part}. These correspond bijectively to
$\bA$-fixed $1$-dimensional ideals in $\cO_{\C^3}$.  An $\bA$-fixed
ideal sheaf on a general toric $X$ is described by a collection 
$\{\pi_v\}$ of  3-dimensional partitions placed at the
vertices of $\Delta(X)$ that glue along the edges of $\Delta(X)$. 
Additionally, nontrivial legs are not allowed along unbounded edges 
of $\Delta(X)$. These rules are illustrated in Figure \ref{f_Om1Om1}.

\subsubsection{} 

The description of $\PT(X)^\bA$ is very similar, see \cite{PT2}. We discuss
localization on the Hilbert scheme here, because we want to make 
contact with the refined vertex of Iqbal, Kozcaz, and Vafa \cite{IKV}. 

Since $\Hilb(X)$ is never compact, Proposition \ref{p_rig} may not 
be directly applied to it. However, if $\PT(X)$ is compact and if we 
assume the conjectural formula \eqref{ePTDT}, then the only chamber 
dependence of the refined invariants comes from the Hilbert scheme 
of points of $X$.  

We have 
$$
\chi\left(X, \cK_X^{1/2} \otimes 
(TX + \cK_X - T^*X- \cK_X^*)\right) = \sum_{x\in X^\bT} 
\rho(T_x X)
$$
and hence by Conjecture \ref{c_DT0} we have 
\begin{multline}
\lim_{z\to 0} \,
\chi\left(\Hilb(X,\textup{points}),\tO_\textup{vir}\right) \Big|_{t\sigma(z)}= \\
\Sd \frac{q}{(1-q\bk^{1/2})(1-q\bk^{-1/2})} 
\sum_{x\in X^T} \left(-\bk^{1/2}\right)^{\sind_\sigma(x)}\label{atoinf} \,,
\end{multline}
for any homomorphism $\sigma: \Ct \to \bA$, where 
$$
\sind_\sigma(x) = \dim (T_x X)_+- \dim (T_x X)_-  = \pm 1
$$
is the index of $x\in X^T$ with respect to $\sigma$. 
The sum over $x$ in \eqref{atoinf} is an equivariant analog of the 
Poincar\'e polynomial of $X$; it jumps across the walls in $\Lie A$ 
dual to the noncompact edges of toric diagram. 

\begin{figure}[!htbp]
  \centering
   \rotatebox{-90}{\scalebox{0.6}{\includegraphics{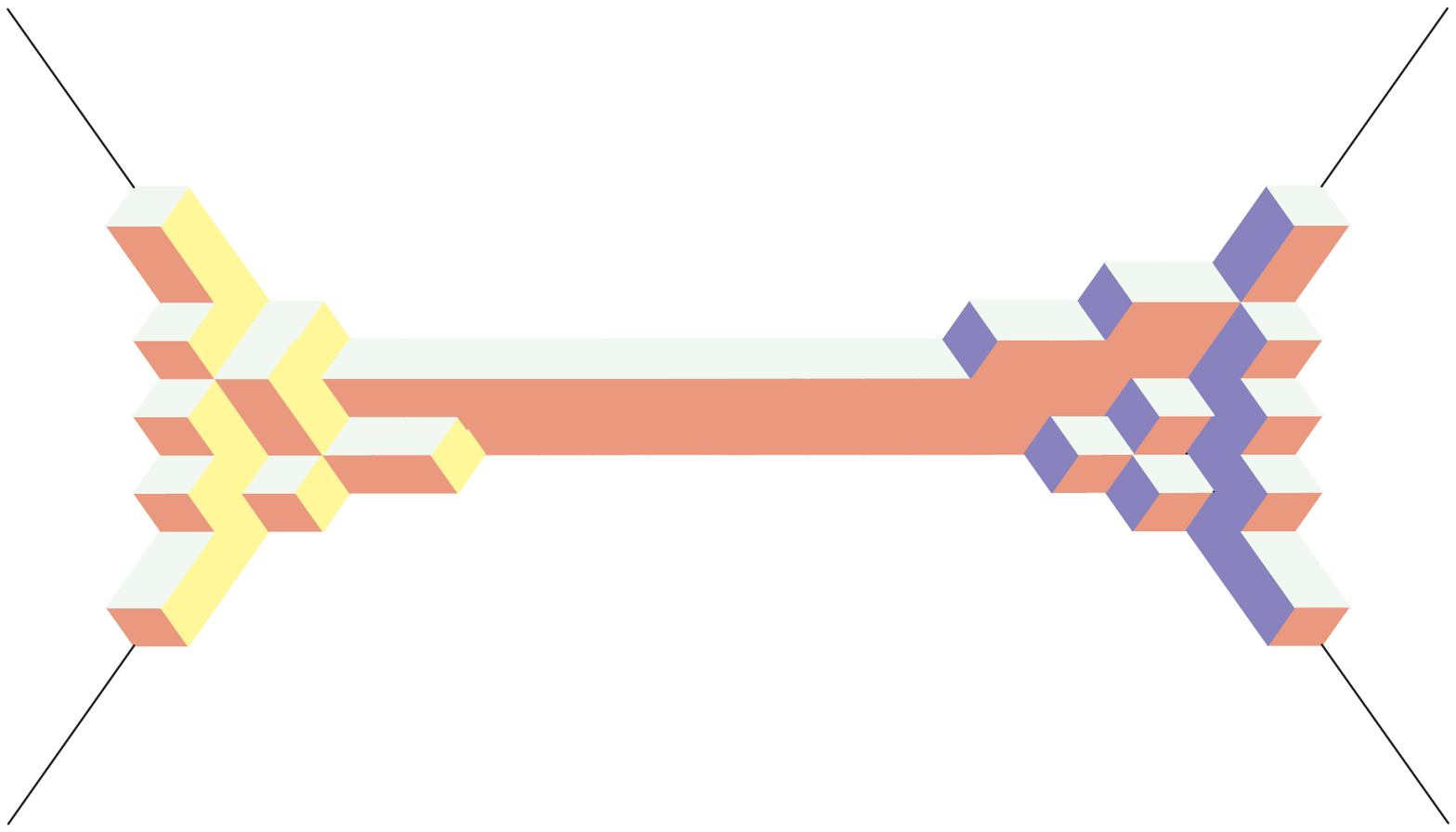}}}
 \caption{An $\bA$-fixed ideal sheaf on $\cO(-1)\oplus\cO(-1)\to \pP^1$}
  \label{f_Om1Om1}
\end{figure}

\subsection{Virtual tangent spaces at fixed points} 

\subsubsection{}

For a $\bT$-fixed ideal sheaf $\cI$ on a 3-fold $X$ we denote 
$$
\cN(\cI) = \chi(\cO_X) -
  \chi(\cI,\cI) \,. 
$$
For $\cI\in \Hilb(X)^\bT$ this is the virtual tangent (and virtual normal) 
space at $\cI$.  
For a $1$-dimensional monomial ideal 
$$
I \subset \C[x_1,x_2,x_3] = \cO_{\C^3}
$$
the character of $\cN(I)$ is well-defined as an element of 
$$
\cN(I) \in \Q(\bT) \supset K_\bT(\pt) \,. 
$$
In fact, the only poles of $\tr_{\cN(I)}$ are first-order poles 
along the weights $t_i$ of the directions of infinite legs, see
\cite{mnop}. For any $\cI$ we have 
$$
\cN(\cI) = \sum_{x \in X^\bT}  \cN(I_x) \,, 
$$
where $I_x$ is the restriction of $\cI$ to the toric chart at $x$, see \cite{mnop}. 

\subsubsection{}

The 
residue of $\tr_{\cN(I)}$ at $t_1=1$ only depends on the saturation 
$$
I^\textup{leg}_1 \supset I 
$$
of $I$ with respect to $x_1$. Combinatorially, $I^\textup{leg}_1$
corresponds to a pure infinite leg in the direction of $x_1$ with the 
same cross-section $\lambda$ as $I$. We have 
$$
\cN(I^\textup{leg}_1) = \C[x_1] \otimes 
T_{I_\lambda} \Hilb(\C^2)\,,
$$
where $I_\lambda \subset \C[x_2,x_3]$ is the corresponding monomial
ideal. 

 If $e$ is an edge of $\Delta(X)$ that joins two vertices $x$
and $x'$, then the corresponding saturations $I^e_{x}$ and $I^e_{x'}$ 
glue to form an ideal sheaf $\cI^e$. Its deformations and obstructions 
are easy to understand using the formula 
$$
\cN(\cI^e) = \cN(I^e_{x}) + \cN(I^e_{x'}) 
$$
and the explicit description of $T_{I_\lambda} \Hilb(\C^2)$ in terms
of arms and legs of the squares of $\lambda$. 

\subsubsection{}

We set 
$$
\cNr(I)=\cN(I) - \sum_{i=1}^3 \cN(I^\textup{leg}_i) \in K_\bT(\pt) \,. 
$$
It satisfies
$$
\cNr(I) = - \C(\bk) \otimes \cNr(I)^\vee \,. 
$$
By construction, 
$$
\cN(\cI) = \sum_{\textup{edges $e$}} \cN(\cI^e) + 
\sum_{\textup{vertices $x$}}  \cNr(I_x) \,.
$$
Since the first term here is explicit, we focus on the vertex
contribution. 

\subsubsection{The K-theoretic vertex}\label{s_K_vertex}

{}From Proposition \ref{p_rho} we have the following formula
for the localization vertex: 
\begin{equation}
\bV(\lambda,\mu,\nu) = \sum_{\pi} (-q)^{|\pi|} \, 
\rho\left(\cNr(I)_+\right)\,, 
\label{full_vert}
\end{equation}
where the sum is over all $3$-legged partitions $\pi$ ending 
on given triple $(\lambda,\mu,\nu)$ of $2$-dimensional partitions, 
$I\in \cO_{\C^3}$ is the corresponding ideal, and the size of 
an infinite partition $\pi$ is defined as a rank
$$
|\pi| = \rk \left(-2\cO-I+\sum I^\textup{leg}_i\right) 
$$
of a finite-dimensional virtual $\bT$-module. Note this size may be 
negative. The decomposition
$$
\cNr(I) = \cNr(I)_+ - \C(\bk) \otimes \cNr(I)^\vee_{+} 
$$
is with respect to sign of the weights on $d\sigma\in \Lie \bA$, 
where $\sigma: \Ct \to \bA$ is generic $1$-parameter subgroup. 
The product $\rho\left(\cNr(I)_+\right)$ is independent of the choice 
of $\sigma$.

\subsubsection{The index vertex}

We now compute the limit of \eqref{full_vert} on $t\sigma(z)\in \bT$ 
as $z\to 0$. This does depend on the choice of $\sigma$ which 
we call the choice of a \emph{slope} since $\bA$ is
2-dimensional. Because we want the slope
to be generic, we  choose is so that 
\begin{equation}
\Q_{>0} \cdot d\sigma  = 
\Q_{>0} \cdot \left(d\sigma_0 + \textup{small perturbation}\right)\label{ray_sigma}
\end{equation}
for some fixed rational slope $d\sigma_0$ and an infinitesimal 
perturbation of a given sign. Given a generic slope like this, we 
define 
\begin{align}
  \label{index_vert}
  \bV^\sigma(\lambda,\mu,\nu) &= \lim_{z\to 0} \, 
\bV(\lambda,\mu,\nu)\big|_{t \sigma(z)} \\
&=
\sum_{\pi} (-q)^{|\pi|} \, 
(-\bk^{1/2})^{\sind_\sigma I} \,, \notag 
\end{align}
where
\begin{equation}
\sind_\sigma I = \rk \cNr(I)_+ \,. \label{slope_index}
\end{equation}
This jumps as the sign of the small perturbation in \eqref{ray_sigma} 
changes, that is, as the ray $\Q_{>0} \cdot d\sigma$ crosses from one side 
of $\Q_{>0} \cdot d\sigma_0$ to the other.  This wall-crossing is analyzed
in \cite{AO}.

\subsection{The refined vertex}

\subsubsection{} 

We call a slope $\sigma$ \emph{preferred} if $\sigma_0$ fixes one of
the coordinate axes. By convention, we choose this to be the $x_3$-axis 
which we plot vertically. In this section, we show that for
preferred
slopes the index vertex specializes to the refined vertex of \cite{IKV}. 

For a general slope, the index \eqref{slope_index} is a quite
complicated function of a 3-dimensional partition because the
character of $\cNr(I)$ 
depends quadratically on the character of $I$ itself. For preferred
slopes, however, there is a big cancellation in $\rk \cNr(I)_+$ 
and the dependence
becomes linear, that is, the index may be computed a certain 
single sum over the boxes $\bx\in\pi$. 

\subsubsection{}

Let $\pi$ be a 3-dimensional partition, let $\pi^\textup{leg}_3$ its 
leg in the preferred direction, and let $\lambda$ the corresponding
$2$-dimensional partition. We view the diagram as a collection of 
squares $\lambda \subset \R_{\ge 0}^2$  in the plane 
and denote by 
$$
f_\lambda = \textup{boundary} \left(\R_{\ge 0}^2 \setminus \lambda
\right)
$$
the \emph{profile} of $\lambda$. This is a zig-zag line going from 
$x_1$-axis to the $x_2$-axis. We label its two possible slopes by 
the corresponding variable $x_i$ and call its corners peaks and
valleys, so that that 
\begin{equation*}
  f_\varnothing = \{x_1\ge 0, x_2=0\} \cup \{x_1= 0, x_2\ge 0\} 
\end{equation*}
has one valley and no peaks.  See \eqref{profile} below for a more
formal definition. 

\subsubsection{}

There is a natural projection 
$$
p: \R^3 \to f_\lambda
$$
along the $(1,1,0)$ and $(0,0,1)$ directions. For a box $\bx\in\pi$ 
we define 
$$
\xi_\lambda(\bx) = 
\begin{cases}
x_1\,, & p(\bx) \in \textup{$x_1$-slope of $f_\lambda$}\,,\\
x_2\,, & p(\bx) \in \textup{$x_2$-slope of $f_\lambda$}\,, \\
x_3^{\pm 1}\,, & p(\bx) \in \textup{peak/valley of $f_\lambda$} \,. 
\end{cases}
$$
This is illustrated in Figure \ref{f_chim}. We define 
\begin{equation}
\Xi(\pi) = \sum_{\bx \in \pi\setminus \pi^\textup{leg}_3}
\xi_\lambda(\bx) - 
\sum_{\bx \in \pi^\textup{leg}_1} 
\xi_\varnothing(\bx)
- 
\sum_{\bx \in \pi^\textup{leg}_2} 
\xi_\varnothing(\bx)\,, \label{Xip}
\end{equation}
which is a finite sum. 

For a preferred slope, the index is computed in the following 

\begin{Theorem}\label{t_pref} 
 If $x_3$ is fixed by $\sigma_0$ then 
$$
\sind_\sigma I = \rk \left(\Xi - \Xi^\vee\right)_{+} \,. 
$$
\end{Theorem}

\noindent 
Of course, one should always bear in mind that the $\bT$-weight of the monomials 
$x_i$ are the \emph{opposite} of the weights of coordinate
directions. 

\subsubsection{}
The proof of Theorem \ref{t_pref} will take several steps. The first 
step is to reduce to the case 
\begin{equation}
\pi^\textup{leg}_1=\pi^\textup{leg}_2=\varnothing  \,. 
\label{trunc12}
\end{equation}
Introduce the following truncation 
$$
I_N = I + (x_1^N,x_2^N) \,,
$$
and the corresponding truncations $I^\textup{leg}_{i,N}$ of 
the saturations $I^\textup{leg}_i$.  The general case of
Theorem \ref{t_pref} is reduced to \eqref{trunc12} by the 
following

\begin{Lemma} We have 
 $$
\sind_\sigma =\rk\left( \cNr(I_N) - \cNr(I^\textup{leg}_{1,N}) - 
 \cNr(I^\textup{leg}_{2,N})\right)_{+} 
$$
for all $N\gg 0$. 
\end{Lemma}

\begin{proof}
 Denote by $Z$ the spectrum of $\cO/I$ so that 
$$
\cO_{Z_N} = \cO_Z-x_1^N \cO_{Z^\textup{leg}_1} - 
x_2^N \cO_{Z^\textup{leg}_2} \,. 
$$
We compute 
\begin{multline}
 \cNr(I) - \cNr(I_N) + \cNr(I^\textup{leg}_{1,N}) +  
 \cNr(I^\textup{leg}_{2,N})  = \\
-\chi\left(\cO_{Z_N\setminus Z^\textup{leg}_{1,N}},
x_1^N \cO_{Z^\textup{leg}_{1}}\right) - 
\chi\left(\cO_{Z_N\setminus Z^\textup{leg}_{2,N}},
x_2^N \cO_{Z^\textup{leg}_{2}}\right) \label{expandN} \\ - 
\chi\left(x_1^N \cO_{Z^\textup{leg}_{1}}, 
x_2^N \cO_{Z^\textup{leg}_{2}}\right) - \dots 
\end{multline}
where dots stand for 3 more terms obtained by reversing 
the order of entries in Euler characteristic. Since the supports 
of all pairs in \eqref{expandN} extend along different coordinate
axes, each Euler characteristic is a zero-dimensional 
element of $K_\bT(\pt)$ shifted by a large nontrivial weight 
of $\sigma_0$. Therefore, it makes no contribution to the 
index. 
\end{proof}

\subsubsection{The balance lemma}\label{s_bal} 

One can think about the situation \eqref{trunc12} a bit more 
abstractly. Let $C$ be a 1-dimensional 
component of $X^{\sigma_0}$ for some general 3-fold $X$, the case at 
hand being the $x_3$-axis in $X=\C^3$. Let $Z$ be a 
$\bA$-invariant subscheme 
contained in an infinitesimal neighborhood of $C$ and let $E \subset Z$
be the closure of the generic point of $Z$ along $C$, for example 
$$
E = Z^\textup{leg}_3
$$
in our concrete situation. We define 
\begin{equation}
\Xi_Z = \chi\!\left(\cI_E,\cO_{Z\setminus E}
\otimes \left(\pi^* TC\right)^{\otimes{N}}\right)\,,\label{XiZ}
\end{equation}
where $TC$ is the tangent bundle of $C$, $\pi: Z \to C$ is 
the $\bT$-equivariant projection, and
$$
0 \ll N \ll \| \textup{small perturbation} \|^{-1} 
$$
where the small perturbation refers to \eqref{ray_sigma}.

The following technical result compares the deformations 
of $Z$ with the deformations of $E$, showing that 
the index of the difference is linear in $[\cO_{Z}]$. 

\begin{Lemma}\label{l_bal} We have 
  \begin{equation}
    \label{bal}
\rk \left(\cN(Z) - \cN(E)\right)_+ = \rk \left(\Xi_Z - \Xi_Z^\vee\right)_+ \,. 
\end{equation}
\end{Lemma}

\noindent 
The proof of this lemma is given in the Appendix. It is clear 
that the Lemma proves the theorem modulo checking that 
the definition \eqref{XiZ} specializes to the formula \eqref{Xip}.

\subsubsection{Conclusion of the proof}\label{s_chim} 

We are now back in the case $X=\C^3$, 
$C$ is the $x_3$-axis, and $I_E=I^\textup{leg}_3$ is generated by 
monomials in $x_1$ and $x_2$. Concretely, 
$$
I_E = \left(\left\{x_1^i x_2^j\right\}_{(i,j)\notin \lambda}\right)
$$
where $\lambda\in \Z_{\ge 0} \times \Z_{\ge 0}$ is the 
diagram of a partition. For example, the case 
$$
I_E = (x_1^2,x_2^3) 
$$
is depicted in Figure \ref{f_chim}. Visually, one may 
compare $E$ to a chimney in the corner of a room 
and then $Z$ corresponds to a few boxes stacked against
this chimney (in violation of all building regulations), 
as in Figure \ref{f_chim}.

\psset{unit=1 cm}
\begin{figure}[!htbp]
  \centering
   \begin{pspicture}(0,0)(7.5,6)
\rput[lb](0,0){\scalebox{0.4}{\includegraphics{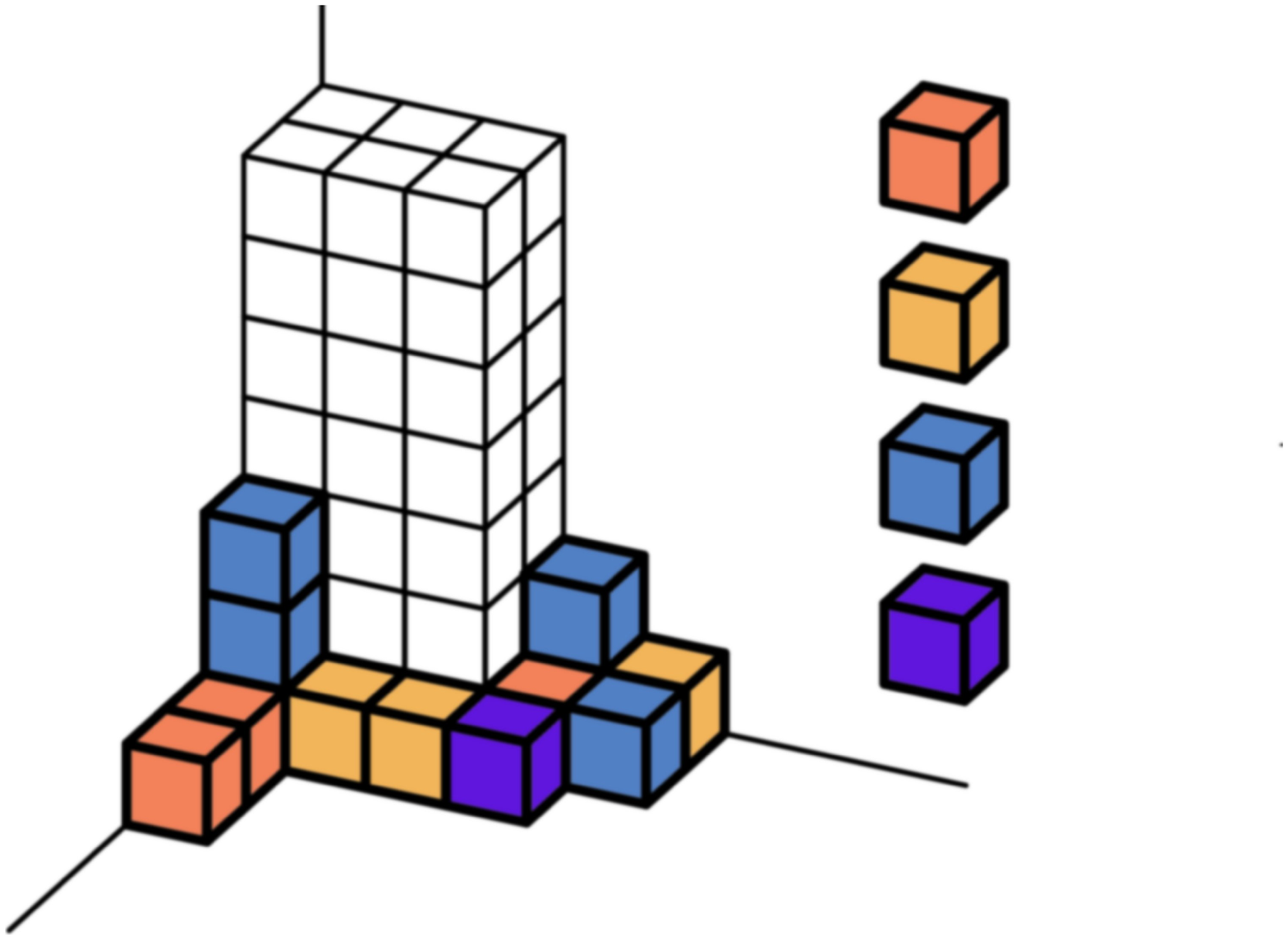}}}
\rput[c](0,0.4){$\displaystyle x_1$}
\rput[c](2.3,5.7){$x_3$}
\rput[c](6,0.7){$x_2$}
\rput[l](6.6,4.95){$=x_1$}
\rput[l](6.6,3.95){$=x_2\sim x_1^{-1}$}
\rput[l](6.6,2.95){$=x_3^{-1}$}
\rput[l](6.6,1.95){$=x_3$}
\end{pspicture}
 \caption{The function $\xi_\lambda(\bx)$}
  \label{f_chim}
\end{figure}

The generators and relations for 
the ideal $I_E$ correspond to the inner and outer
corners of the chimney, that is, to the valleys and the 
peaks of the profile of $\lambda$. In our example, we have generators
of degree $x_1^2$ and $x_2^3$, together with a relation of 
degree $x_1^2 x_2^3$. In general, let us denote 
by 
$$
\gamma_i,\rho_i \in \Z_{\ge 0}^2  \subset \Z^3
$$
the multiindex degrees of generators and relations of $I_E$. 
The we have an equivariant free resolution
\begin{equation}
0 \to \bigoplus_{i=1}^{g-1}  x^{\rho_i}  \cO_{\C^3}
\to \bigoplus_{i=1}^g  x^{\gamma_i}  \cO_{\C^3} \to 
I_E \to 0
\label{free_res}
\end{equation}
where $g$ is the number of the generators (number of inner 
corners). 

Since the $\bA$-weight of $x_3\in T^*C$ is minus the weight 
of $TC$, the contribution of a monomial 
$x^\bx\in \cO_{Z\setminus E}$ to the index of $\Xi_Z$ 
equals the $\sigma$-index
of the following $\bA$-module 
$$
\Xi_\bx= x_3^{-N} \sum_{i=1}^g{x^{\bx-\gamma_i}} + x_3^{N} 
\sum_{i=1}^{g-1}{x^{\rho_i-\bx}} \,. 
$$
This index is computed as follows. We may assume 
$$
d\sigma  =  x_1 \frac{\partial}{\partial x_1} 
- x_2 \frac{\partial}{\partial x_2} \,. 
$$
Consider the function
$$
c((a_1,a_2,a_3)) = a_1 - a_2 \,. 
$$
Clearly, $d\sigma_0 \cdot x^\bx = c(\bx)\, x^\bx$. 
The level sets of $c(\bx)$ 
are the diagonal slices in Figure \ref{f_chim}. Let $f_\lambda(s)$ be 
the profile of $\lambda$, defined by 
\begin{equation}
\tfrac{1}{2} f''_\lambda = \sum \delta_{c(\gamma_i)}
- \sum \delta_{c(\rho_i)}\,,\label{profile}
\end{equation}
together with 
$$
f_\lambda(s) = |s|\,, \quad |s| \gg 0 \,. 
$$
Then 
$$
\sind_\sigma \Xi_\bx = f'_\lambda(c(\bx)) 
$$
extended by left or right continuity depending on the weight of $x_3$. 
This is color-coded in Figure \ref{f_chim}. Clearly, 
$$
\sind_\sigma \Xi_\bx = \xi_\lambda(\bx)
$$
which concludes the proof.

\appendix

\section{Appendix}

\subsection{Proof of the balance lemma}

\subsubsection{} 

Since $\cO_{Z}/\cO_{E}$ is zero-dimensional, the statement of 
Lemma \ref{l_bal} is purely local and we can assume we 
are in the situation of Section \ref{s_chim}.

We begin we the special case $E=\varnothing$, so that 
$\cN(E)=0$. We 
need to show 
\begin{equation}
\cN(Z) \equiv x_3^{-N} \, \chi(\cO_Z) - 
x_3^{N} \, \overline{\chi(\cO_Z)}
\label{bal1}
\end{equation}
for all sufficiently large $N$ and the relation $\equiv$ on
$\bA$-modules
means that they have the same $\sigma$-index. 

\subsubsection{} 

In fact, it suffices to take $N$ so that $x_3^N \in \cI_Z$, 
in other words, that $N$ is larger than the height of the 
stack of boxes corresponding to $Z$. As we will see, 
the difference
between the left-hand side and the right-hand side of
\eqref{bal1} corresponds to the deformations of 
$$
\cF = \pi_{12*} \left(\cI_{Z}/x_3^N \right)\,.
$$
Here $\pi_{12}$ is the projection onto the $(x_1,x_2)$-plane 
and $\pi_{12*}$ means that we view $\cI_{Z}/x_3^N$ as 
a $\cO_{\C^2}=\C[x_1,x_2]$ module, that is as 
a (degenerate) framed rank $N$ instanton on $\C^2$. 

Away from the origin of $\C^2$, $\cF$ is the same as 
$$
\cF_{\varnothing} = \sum_{i=0}^{N-1} x_3^i \, \cO_{\C^2}\,. 
$$
The $\bA$-action on the $x_3$ is thus transformed into 
into the action on the framing of $\cF$. 

\subsubsection{}

Moduli of framed torsion-free sheaves on $\C^2$ is a smooth 
manifold with tangent space
\begin{equation}
\cN_2(Z) = \chi_{\C^2}(\cF_{\varnothing},\cO_Z) 
+  \chi_{\C^2}(\cO_Z,\cF_{\varnothing}) - \chi_{\C^2}(\cO_Z,\cO_Z) \,,
\label{Xi2F}
\end{equation}
where $\chi_{\C^2}$ means we treat all sheaves
as $\bA$-equivariant $\cO_{\C^2}$-modules. In particular, 
$$
\chi_{\C^2}(\cO_0,\cO_0) = (1 - x_1^{-1}) (1- x_2^{-1}) = 
(1-x_3^{-1})^{-1}\, \chi(\cO_0,\cO_0) 
$$
where $0\in \C^2\subset \C^3$ is the origin. It follows 
that 
$$
\cN(Z) = (1-x_3^{-1}) \, \cN_2(Z) + x_3^{-N} \, \chi(\cO_Z) - 
x_3^{N} \, \overline{\chi(\cO_Z)} 
$$
Therefore, \eqref{bal1} is equivalent to showing 
\begin{equation}
\sind_\sigma \, \cN_2(Z) = 0 
\label{bal2}
\end{equation}

\subsubsection{}

The symplectic form $dx_1 \wedge dx_2$ on $\C^2$ induces
a symplectic form 
on instanton moduli and, in particular, 
a symplectic form $\omega_\cF$
on tangent space $\cN_2(\cF)$ to $\cF$. The torus $\bA$
scales this symplectic form with the weight of $x_1 x_2$, 
which is the 
same as the weight of $x_3^{-1}$. Therefore, $\omega_\cF$ pairs
attracting and repelling direction with the exception of 
the weights $0$ and $x_3^{-1}$. Therefore \eqref{bal2}
follows from the following 

\begin{Lemma} 
$$
\cN_2(Z)^{\bA} = 0\,.
$$
\end{Lemma}

\begin{proof}
We have 
$$
\cF = \bigoplus_{i=1}^N x_3^{i-1} \, \cI_{\lambda^{(i)}} 
$$
where $\cI_\lambda \subset \C[x_1,x_2]$ is a monomial ideal 
corresponding to a partition $\lambda$ and
$$
\lambda^{(1)} \supset \lambda^{(2)} \supset \dots \supset 
\lambda^{(N)} \,. 
$$
Therefore 
$$
\cN_2(Z) = \sum_{i,j} 
x_3^{j-i}\left(\chi_{\C^2}(\cO_{\C^2}) - 
\chi_{\C^2}
\left(\cI_{\lambda^{(i)}},\cI_{\lambda^{(j)}}\right)
\right)
$$
The claim then follows from the following Lemma \ref{l_incl}. 
\end{proof}

\subsubsection{}

\begin{Lemma}\label{l_incl} 
Let $\cI,\cJ\in \C[x_1,x_2]$ be monomial ideals. If 
$$
(n <0 \textup{ and } \cI \supset \cJ) \textup{ or }
(n \ge 0 \textup{ and } \cI \subset \cJ) 
$$
then the weight $(x_1 x_2)^n$ 
does not occur in 
$$
\chi_{\C^2}(\cO_{\C^2}) - 
\chi_{\C^2}
\left(\cI,\cJ\right)\,. 
$$
\end{Lemma}

\begin{proof}
By duality, it suffices to consider the case $\cI \supset \cJ$. 
{} From a resolution like \eqref{free_res}, we have 
\begin{equation}
  \supp \, \chi_{\C^2}
\left(\cI,\cJ\right)
\subset 
  \bigcup_{i,j} 
\supp \,\, x^{b_j-a_j} \cO_{\C^2} \,,
\label{suppsupp} 
\end{equation}
where 
$$
\{a_i\},\{b_j\} \subset (\Z_{\ge 0})^2
$$ 
are the degrees of 
generators and relations for $I$ and $J$, respectively, and 
by support of a $\bA$-module we mean the support of its
Fourier transform, that is, the set of weights that occur in it. 

The inclusion $\cI \supset \cJ$ obviously implies 
$$
b_j - a_i \notin (\Z_{<0})^2
$$
and hence all supports in \eqref{suppsupp} are 
disjoint from $(\Z_{<0})^2$. 
\end{proof}

This completes the proof of \eqref{bal1}. 

\subsubsection{}

We now deduce the general case of Lemma \ref{l_bal} from 
its special case \eqref{bal1}. Choose $N$ so large that 
$x_3^N$ annihilates $\cO_{Z\setminus E}$. Expanding 
$\cN$ using 
$$
\left[\cO_{Z}\right] = 
\left[\cO_{Z\setminus E}\right] 
+
\left[\cO_{E}/x_3^N\right] + x_3^N \left[\cO_{E}\right]
$$
we get 
\begin{equation}
\cN(Z) - \cN(E)  = 
  \cN(Z/x_3^N) - \cN(E/x_3^N)  
- \Alt \, x_3^{-N} \chi(\cO_E, \cO_{Z\setminus E} ) \,. 
\end{equation}
where, by definition, 
$$
\Alt V = V - V^\vee 
$$
for any $\bA$-module $V$. 
{}From \eqref{bal1} we conclude
$$
\cN(Z/x_3^N) - \cN(E/x_3^N)  \equiv \Alt 
 x_3^{-N} \chi(\cO_{\C^3}, \cO_{Z\setminus E}) \,.
$$
This concludes the proof.

\newpage 

\noindent
Nikita Nekrasov\\ 
Simons Center for Geometry and Physics\\ 
Stony Brook University, Stony Brook NY 11794-3636, U.S.A.
\footnote{on leave of absence from:
IHES, Bures-sur-Yvette, France, 
ITEP and IITP, Moscow, Russia}
\\
\texttt{nikitastring@gmail.com} 

\vspace{+10 pt}

\noindent 
Andrei Okounkov\\
Department of Mathematics, Columbia University\\
New York, NY 10027, U.S.A.\\

\vspace{-12 pt}

\noindent 
Institute for Problems of Information Transmission\\
Bolshoy Karetny 19, Moscow 127994, Russia\\

\vspace{-12 pt}

\noindent 
Laboratory of Representation
Theory and Mathematical Physics \\
Higher School of Economics \\ 
Myasnitskaya 20, Moscow 101000, Russia

\end{document}